\documentclass[a4paper]{article}

\usepackage[latin1]{inputenc} 
\usepackage{graphicx}
\usepackage[pdftex,bookmarks]{hyperref}
\usepackage{amsmath}    
\usepackage{amssymb}
\usepackage{amsmath,amsthm}
\usepackage{textcomp}
\usepackage[all]{xy}
\usepackage{geometry}

\newtheorem{teo}{Theorem}[section]
\newtheorem{defi}{Definition}[section]

\newtheorem{cor}{Corollary}[section]
\newtheorem{prop}{Proposition}[section]

\newtheorem{rem} {Remark}[section]

\newcommand{\nn}{\nonumber}

\DeclareMathOperator{\im}{im}

\DeclareMathOperator{\dvol}{dvol}

\DeclareMathOperator{\supp}{supp}

\DeclareMathOperator{\id}{Id}
\DeclareMathOperator{\vol}{vol}
\DeclareMathOperator{\reg}{reg}
\DeclareMathOperator{\sing}{sing}
\DeclareMathOperator{\End}{End}
\DeclareMathOperator{\re}{Re}
\DeclareMathOperator{\si}{sign}
\DeclareMathOperator{\Hom}{Hom}
\DeclareMathOperator{\tr}{tr}
\DeclareMathOperator{\Tr}{Tr}
\DeclareMathOperator{\op}{op}
\DeclareMathOperator{\inte}{int}
\DeclareMathOperator{\cod}{cod}

\DeclareMathOperator{\depth}{depth}

\title{\huge \bf  Sobolev spaces and Bochner Laplacian  on  complex projective varieties and stratified pseudomanifolds}
\author{Francesco Bei  \bigskip \\
Institute f\"ur Mathematik, Humboldt Universit\"at zu Berlin,\\ E-mail addresses: \ bei@math.hu-berlin.de\     francescobei27@gmail.com }
\date{}

\begin{document}

\maketitle
\begin{abstract}
 Let $V\subset \mathbb{C}\mathbb{P}^n$ be an irreducible complex projective variety of complex dimension $v$ and let $g$ be the K\"ahler metric on $\reg(V)$, the regular part of $V$,  induced by the Fubini Study metric of $\mathbb{C}\mathbb{P}^n$. In \cite{LT} Li and Tian proved that $W^{1,2}_0(\reg(V),g)=W^{1,2}(\reg(V),g)$, that the natural inclusion $W^{1,2}(\reg(V),g)\hookrightarrow L^2(\reg(V),g)$ is a compact operator and that the heat operator associated to the Friedrich extension of the scalar Laplacian $\Delta_0:C^{\infty}_c(\reg(V))\rightarrow C^{\infty}_c(\reg(V))$, that is  $e^{-t\Delta_0^{\mathcal{F}}}:L^2(\reg(V),g)\rightarrow L^2(\reg(V),g)$, is a trace class operator. The goal of this paper is to provide an extension of the above  result to the case of Sobolev spaces of sections and symmetric Schr\"odinger type operators with potential bounded from below  where the underling Riemannian manifold is the regular part of a complex projective variety endowed with the Fubini-Study metric or the regular part of a  stratified pseudomanifold endowed with  an iterated edge metric.
\end{abstract}
\vspace{1 cm}

\vspace{0.5 cm}


\tableofcontents

\section*{Introduction}
Complex projective varieties endowed with the Fubini-Study metric as well as stratified pseudomanifolds with an iterated edge metric are important examples of singular spaces with a rich interplay between topological and analytic questions. An important topic in this setting  is certainly provided by the heat operator and its properties.  Many papers during the last thirty years have been devoted to explore this subject. Without any goal of completeness we can mention here the seminal paper of Cheeger  \cite{JCH}, where the study of the heat kernel on stratified pseudomanifolds has been initiated,  \cite{BS2}, \cite{BSE2}, \cite{BS}, \cite{MaVe}, \cite{EMO} where the heat operator on manifolds with conical singularities and on manifolds with edges is studied,  \cite{BRLU}, \cite{LT}  \cite{MasNag} and \cite{Pa} that  deal with the heat operator on complex projective varieties and so on. In particular in \cite{LT}, generalizing the results established in \cite{MasNag} and \cite{Pa},  Li and Tian  proved, without any assumptions on the singularities of $V$, that $W^{1,2}_0(\reg(V),g)=W^{1,2}(\reg(V),g)$ (in other words the $L^2$-Stokes Theorem holds for functions), that the natural inclusion $W^{1,2}(\reg(V),g)\hookrightarrow L^2(\reg(V),g)$ is a compact operator and that the heat operator associated to the Friedrich extension of the scalar Laplacian $\Delta_0:C^{\infty}_c(\reg(V))\rightarrow C^{\infty}_c(\reg(V))$, that is  $e^{-t\Delta_0^{\mathcal{F}}}:L^2(\reg(V),g)\rightarrow L^2(\reg(V),g)$, is a trace class operator.\\ In this paper we are interested to extend the result of Li and Tian to the case of Sobolev spaces of sections and to symmetric Schr\"odinger type operators with potential bounded form below where the underling Riemannian manifold is the regular part of a complex projective variety endowed with the Fubini-Study metric or the regular part of a  stratified pseudomanifold with an iterated edge metric.\\ Let us go more into the details explaining the structure of the paper. The first section is devoted to the background material. We recall briefly the definition of $L^p$ spaces, maximal and minimal extension of a differential operator and the notion of Sobolev space associated to a connection. In particular, given an open and possibly incomplete Riemannian manifold $(M,g)$ with a vector bundle $E$ endowed with a metric $h$, we will consider the spaces $W^{1,2}(M,E)$ and $W^{1,2}_0(M,E)$. The former  is the space of sections  $s\in L^2(M,E)$ such that $\nabla s$, applied in the distributional sense, lies in $L^2(M,T^*M\otimes E)$. The latter is defined as the graph closure of $\nabla:L^2(M,E)\rightarrow L^2(M,T^*M\otimes E)$ with core domain $C^{\infty}_c(M,E)$, the space of smooth sections with compact support. In the  second section we recall   Kato's inequality and then we provide some results about the  domination of semigroups. In particular, under some additional assumptions, we give a proof of the domination of the heat semigroups on a possibly incomplete Riemannian manifold, which is based on Kato's inequality. The third section contains some general results concerning Sobolev spaces of sections and symmetric Schr\"odinger operator with potential bounded from below. The fourth section concerns irreducible complex projective varieties. The main result of its first part is the following theorem:

\begin{teo}
\label{pianopianozx}
Let $V\subset \mathbb{C}\mathbb{P}^n$ be an irreducible complex projective variety of complex dimension $v$. Let $E$ be a vector bundle over $\reg(V)$ and let $h$ be a metric on $E$, Hermitian if $E$ is a complex vector bundle, Riemannian if $E$ is a real vector bundle. Let $g$ be the K\"ahler  metric on $\reg(V)$ induced by the Fubini-Study metric of  $\mathbb{C}\mathbb{P}^n$. Finally let $\nabla:C^{\infty}(\reg(V),E)\rightarrow C^{\infty}(\reg(V),T^*\reg(V)\otimes E)$ be a metric connection. We have the following properties:
\begin{itemize}
\item $W^{1,2}(\reg(V),E)=W^{1,2}_0(\reg(V),E)$.
\item Assume that $v>1$. Then there exists a continuous inclusion $W^{1,2}(\reg(V),E)\hookrightarrow L^{\frac{2v}{v-1}}(\reg(V),E)$.
\item Assume that $v>1$. Then the inclusion $W^{1,2}(\reg(V),E)\hookrightarrow L^2(\reg(V),E)$ is a compact operator.
\end{itemize}
\end{teo}

The proof of this theorem lies essentially on a combination of Kato's inequality, Sobolev inequality and the existence of a suitable sequence of cut-off functions. Moreover from Theorem \ref{pianopianozx} we have the following application: for a large class of first order differential operators $D:C^{\infty}_c(\reg(V),E)\rightarrow C^{\infty}_c(\reg(V),F)$, see Theorem \ref{drivevz} for the definition, which includes for instance the de Rham differential $d_k$, the Dolbeault operator $\overline{\partial}_{p,q}$ and Dirac type operators, we have the following inclusion:
\begin{equation}
\label{caldare}
\mathcal{D}(D_{\max})\cap L^{\infty}(\reg(V),E)\subset \mathcal{D}(D_{\min}).
\end{equation}
In the second part of the fourth section we consider Schr\"ondinger type operators $P:C^{\infty}_c(\reg(V),E)\rightarrow C^{\infty}_c(\reg(V),E)$,  $P:=\nabla^t\circ \nabla+L$, which are positive and formally self-adjoint. We study some properties of $e^{-tP^{\mathcal{F}}}$, the heat operator associated to the Friedrich extension of $P$. Our mean result is the following: 

\begin{teo}
\label{skernelzx}
Let $V$, $E$, $g$, $h$, and $\nabla$ be as in Theorem \ref{pianopianozx}. Let $P:=\nabla^t\circ \nabla +L$, $$P:C^{\infty}_c(\reg(V),E)\rightarrow C_c^{\infty}(\reg(V),E)$$ be a Schr\"odinger type operator with $L\in C^{\infty}(\reg(V),\End(E))$. Assume that:
\begin{itemize}
\item $P$ is symmetric and  positive.
\item There is a constant $c\in \mathbb{R}$ such that, for each $s\in C^{\infty}(\reg(V),E)$, we have  $h(Ls,s)\geq ch(s,s).$
\end{itemize}
Consider $P^{\mathcal{F}}:L^2(\reg(V),E)\rightarrow L^2(\reg(V),E)$  and   $\Delta_0^{\mathcal{F}}:L^2(\reg(V),g)\rightarrow L^2(\reg(V),g)$  respectively   the Friedrich extension of $P$ and  the Friedrich extension of $\Delta_0:C^{\infty}_c(\reg(V))\rightarrow C^{\infty}_c(\reg(V))$.
Then the heat operator associated to $P^{\mathcal{F}}$ $$e^{-tP^{\mathcal{F}}}:L^{2}(\reg(V),E)\longrightarrow L^2(\reg(V),E)$$  
is a trace class operator and its trace satisfies the following inequality: 
\begin{equation}
\label{marzzx}
\Tr(e^{-tP^{\mathcal{F}}})\leq me^{-tc}\Tr(e^{-t\Delta_0^{\mathcal{F}}})
\end{equation}
where $m$ is the rank of the vector bundle $E$.
\end{teo}

This theorem is proved applying the results about the domination of semigroups recalled in the second section.  In the remaining part of the forth section we discuss some corollaries of Theorem \ref{skernelzx}. In particular we get an asymptotic inequality for the eigenvalues of $P^{\mathcal{F}}$ and an estimate for the  trace  $\Tr(e^{-tP^\mathcal{F}})$ when $t\in (0,1)$. Moreover we point out that these results apply to $(\nabla^t\circ \nabla)^{\mathcal{F}}:L^2(\reg(V),E)\rightarrow L^2(\reg(V),E)$, the Friedrich extension of the Bochner Laplacian $\nabla^t\circ \nabla:C^{\infty}_c(\reg(V),E)\rightarrow C^{\infty}_c(\reg(V),E)$. Finally, another application of Theorem \ref{skernelzx}, is provided by the extension of Cor. 5.5 of \cite{LT} to our setting. More precisely we prove the following result:
\begin{teo}
\label{lowerz}
There exists a positive constant $\gamma=\gamma(d,n,m)$, that is $\gamma$ depends only on the dimension of the ambient space $\mathbb{C}\mathbb{P}^n$, on the degree $d$ and  on the rank $m$, such that for every irreducible complex projective variety $V\subset \mathbb{C}\mathbb{P}^n$ of degree $d$, for every vector bundle $E$ on $\reg(V)$ of rank $m$ endowed with an arbitrary metric $h$ and for every Schr\"odinger type operator $P:C^{\infty}_c(\reg(V),E)\rightarrow C^{\infty}_c(\reg(V),E)$ as in Theorem \ref{skernelzx} with  $L\geq 0$, the $(md)$-th eingenvalue of $P^{\mathcal{F}}$, that is $\lambda_{md}$, satisfies the following inequality:
\begin{equation}
\label{vicvic}
0<\gamma \leq \lambda_{md}. 
\end{equation}
\end{teo}

The fifth section contains  applications to stratified pseudomanifolds. We start recalling the basic definitions and properties and then we  prove analogous results to those proved in the fourth section. More precisely we have the following  theorem: 

\begin{teo}
\label{pianopianofv}
Let $X$ be a compact, smoothly Thom-Mather stratified pseudomanifold of dimension $m$.  Consider on $\reg(X)$ an iterated edge metric $g$.  Let $E$ be a vector bundle over $\reg(X)$ and let $h$ be a metric on $E$, Riemannian if $E$ is a real vector bundle, Hermitian if $E$ is a complex vector bundle.  Finally let $\nabla:C^{\infty}(\reg(X),E)\rightarrow C^{\infty}(\reg(V),T^*\reg(X)\otimes E)$ be a metric connection. We have the following properties:
\begin{itemize}
\item $W^{1,2}(\reg(X),E)=W^{1,2}_0(\reg(X),E)$.
\item Assume that $m>2$. Then there exists a continuous inclusion $W^{1,2}(\reg(X),E)\hookrightarrow L^{\frac{2m}{m-2}}(\reg(X),E)$.
\item Assume that $m>2$. Then the inclusion $W^{1,2}(\reg(X),E)\hookrightarrow L^2(\reg(X),E)$ is a compact operator.
\end{itemize}
\end{teo}

Similarly to \eqref{caldare},   using Theorem \eqref{pianopianofv},  we derive the following conclusion: for a large class of first order differential operators $D:C^{\infty}_c(\reg(V),E)\rightarrow C^{\infty}_c(\reg(V),F)$, see Theorem \ref{drivevz} for the definition, which includes for instance the de Rham differential $d_k$ and Dirac type operators, we have the following inclusion:
\begin{equation}
\label{caldarebelle}
\mathcal{D}(D_{\max})\cap L^{\infty}(\reg(X),E)\subset \mathcal{D}(D_{\min}).
\end{equation}

Furthermore we prove the following theorem:

\begin{teo}
\label{skernelfv}
Let $X$, $E$, $g$, $h$, and $\nabla$ be as in Theorem \ref{pianopianofv}. Let $$P:=\nabla^t\circ \nabla +L,\ P:C^{\infty}_c(\reg(X),E)\rightarrow C_c^{\infty}(\reg(X),E)$$ be a Schr\"odinger type operator with $L\in C^{\infty}(\reg(X),\End(E))$. Assume that:
\begin{itemize}
\item $P$ is symmetric and  positive.
\item    There is a constant $c\in \mathbb{R}$ such that, for each $s\in C^{\infty}(\reg(X),E)$, we have  $$h(Ls,s)\geq ch(s,s).$$
\end{itemize}
Let $P^{\mathcal{F}}:L^2(\reg(X),E)\rightarrow L^2(\reg(X),E)$ be the Friedrich extension of $P$ and let  $\Delta_0^{\mathcal{F}}:L^2(\reg(X),g)\rightarrow L^2(\reg(X),g)$  be the Friedrich extension of $\Delta_0:C^{\infty}_c(\reg(X))\rightarrow C^{\infty}_c(\reg(X))$. Then the heat operator associated to $P^{\mathcal{F}}$ $$e^{-tP^{\mathcal{F}}}:L^{2}(\reg(X),E)\longrightarrow L^2(\reg(X),E)$$  
is a trace class operator and its trace satisfies the following inequality: 
\begin{equation}
\label{marzfv}
\Tr(e^{-tP^{\mathcal{F}}})\leq re^{-tc}\Tr(e^{-t\Delta_0^{\mathcal{F}}})
\end{equation}
where $r$ is the rank of the vector bundle $E$.
\end{teo}

Finally, in the last part of the fifth section, using Theorem \ref{skernelfv}, we derive some consequences for the operator $P^{\mathcal{F}}$, such as discreteness, an asymptotic inequality for its eigenvalues and an estimate for the trace  $\Tr(e^{-tP^{\mathcal{F}}})$ when $t\in (0,1)$. Moreover, analogously to the previous section, we point out that these results apply to $(\nabla^t\circ \nabla)^{\mathcal{F}}:L^2(\reg(X),E)\rightarrow L^2(\reg(X),E)$, that is the Friedrich extension of the Bochner Laplacian $\nabla^t\circ \nabla:C^{\infty}_c(\reg(X),E)\rightarrow C^{\infty}_c(\reg(X),E)$.

\vspace{1 cm}

\noindent\textbf{Acknowledgments.}   I wish to thank Pierre Albin, Jochen Br\"uning, Simone Cecchini, Eric Leichtnam, Paolo Piazza  and Jean Ruppenthal for interesting comments and useful discussions. This research has been financially supported by the SFB 647 : Raum-Zeit-Materie.

\section{Background material}

The aim of this section is to  recall briefly some basic notions about $L^p$-spaces,  Sobolev spaces and differential operators and then to prove some propositions that we will use often in the rest of the paper. We refer to \cite{TAU}, \cite{BGV}, \cite{GYA},   or the appendix in \cite{KWE} for a thorough discussion about this background material.
Let $(M,g)$  be an open and possibly incomplete Riemannian manifold of dimension $m$. Let $E$ be a vector bundle over $M$ of rank $k$ and let $h$ be a metric on $E$, Hermitian if $E$ is a complex vector bundle, Riemannian if $E$ is a real vector bundle. Let $\dvol_g$ be the one-density associated to $g$.  We consider $M$ endowed with the {\em Riemannian measure} as in \cite{GYA} pag. 59 or \cite{BGV} pag. 29.  A section $s$ of $E$ is said {\em measurable} if, for any trivialization $(U,\phi)$ of $E$, $\phi(s|_U)$ is given by a $k$-tuple of measurable functions.  Given a measurable section $s$ let $|s|_h$ be defined as $|s|_h:= (h(s,s))^{1/2}$. Then for every $p$, $1\leq p< \infty$ we can define $L^{p}(M,E)$ as the  space  of measurable sections $s$ such that    $$\|s\|_{L^{p}(M,E)}:=(\int_{M}|s|_h^p\dvol_g)^{1/p}<\infty.$$
For each $p\in [1, \infty)$ we have a Banach space,  for each $p\in (1, \infty)$  we get a reflexive Banach space and  in the case $p=2$ we have a Hilbert space whose inner product is $$\langle s, t \rangle_{L^2(M,E)}:= \int_{M}h(s,t)\dvol_g.$$ Moreover $C^{\infty}_c(M,E)$,  the space of smooth sections with compact support,  is dense in $L^p(M,E)$ for $p\in [1,\infty).$ Finally $L^{\infty}(M,E)$ is defined as the space of measurable sections whose {\em essential\ supp} is bounded, that is the space of measurable sections $s$ such that $|s|_h$ is bounded almost everywhere. Also in this case we get a Banach space. Clearly all the spaces we defined so far depend on $M$, $E$, $h$ and $g$ but in order to have a lighter notation we prefer to write $L^p(M,E)$ instead of $L^p(M,E,h,g)$. In the case $E$ is the trivial bundle $M\times \mathbb{R}$ we will write $L^p(M,g)$ while for the k-th exterior power of the cotangent bundle, that is $\Lambda^kT^*M$, we will write as usual  $L^p\Omega^k(M,g)$.\\
Let now $F$ be another vector bundle over $M$ endowed with a metric $\rho$. Let $P: C^{\infty}_c(M,E)\longrightarrow  C^{\infty}_c(M,F)$ be a differential operator of order $d\in \mathbb{N}$. Then the {\em formal adjoint} of $P$ $$P^t: C^{\infty}_c(M,F)\longrightarrow C^{\infty}_c(M,E)$$ is the differential operator uniquely characterized by the following property: for each $u\in C^{\infty}_c(M,E)$ and for each $v\in C^{\infty}_c(M,F)$ we have  $$\int_{M}h(u, P^tv)\dvol_g=\int_M\rho(Pu,v)\dvol_g.$$ We can look at $P$ as an unbounded, densely defined and closable operator  acting between $L^2(M,E)$ and $L^2(M,F)$. In general $P$ admits several different closed extensions all included between the minimal and the maximal one. We recall now their  definitions.
The domain of the {\em maximal extension} of $P:L^2(M,E)\longrightarrow L^2(M,F)$ is defined as\\ 
\begin{align}
\label{ner}
& \mathcal{D}(P_{\max}):=\{s\in L^{2}(M,E): \text{there is}\ v\in L^2(M,F)\ \text{such that}\ \int_{M}h(s,P^t\phi)\dvol_g=\\
 \nn &=\int_{M}\rho(v,\phi)\dvol_g\ \text{for each}\ \phi\in C^{\infty}_c(M,F)\}.\ \text{In this case we put}\ P_{\max}s=v.
\end{align} In other words the maximal extension of $P$ is the one defined in the {\em distributional sense}.\\The domain of the {\em minimal extension} of $P:L^2(M,E)\longrightarrow L^2(M,F)$ is defined as\\ 
\begin{align}
\label{spinaci}
& \mathcal{D}(P_{\min}):=\{s\in L^{2}(M,E)\ \text{such that there is a sequence}\ \{s_i\}\in C_c^{\infty}(M,E)\ \text{with}\ s_i\rightarrow s\\ \nn & \text{in}\ L^{2}(M,E)\ \text{and}\ Ps_i\rightarrow w\ \text{in}\ L^2(M,F)\ \text{to some }\ w\in L^2(M,F)\}.\ \text{We put}\ P_{\min}s=w.
\end{align} Briefly the minimal extension of $P$ is the closure of $C^{\infty}_c(M,E)$ under the graph norm $\|s\|_{L^2(M,E)}+\|Ps\|_{L^2(M,F)}$.  It is immediate to check that 
\begin{equation}
\label{emendare}
P_{\max}^*=P^t_{\min}\ \text{and that}\  P_{\min}^*=P^t_{\max}
\end{equation}
 that is  $P^t_{\max/\min}:L^2(M,F)\rightarrow L^2(M,E)$ is the Hilbert space adjoint of $P_{\min/\max}$ respectively. Moreover we have the following two $L^2$-orthogonal decomposition for $L^2(M,E)$
\begin{equation}
\label{fantaghi}
L^2(M,E)=\ker(P_{\min/\max})\oplus \overline{\im(P^t_{\max/\min})}.
\end{equation}
We have the following properties that we will use often later:

\begin{prop}
\label{partiro}
Let $(M,g)$,  $E$ and $F$ be as above.  Let $P:C^{\infty}_c(M,E)\rightarrow C^{\infty}_c(M,F)$ be a differential operator such that  $P^t\circ P:C^{\infty}_c(M,E)\rightarrow C^{\infty}_c(M,E)$ is elliptic. Let $\overline{P}:L^2(M,E)\rightarrow L^2(M,F)$ be a closed extension of $P$. In particular $\overline{P}$ might be $P_{\max}$ or $P_{\min}$. Let $\overline{P}^*$ be the Hilbert space adjoint of $\overline{P}$.  Then $C^{\infty}(M,E)\cap \mathcal{D}(\overline{P}^*\circ \overline{P})$ is dense in $\mathcal{D}(\overline{P})$ with respect to the graph norm of $\overline{P}$. In particular we have that  $C^{\infty}(M,E)\cap \mathcal{D}(P_{\max/\min})$ is dense in $\mathcal{D}(P_{\max/\min})$ with respect to the graph norm of $P_{\max/\min}$.
\end{prop}

\begin{proof}
We start pointing out that, if $\mathcal{D}(\overline{P})$ is the domain of $\overline{P}$ and $\mathcal{D}(\overline{P}^*)$ is the domain of $\overline{P}^*$, then  $\mathcal{D}(\overline{P}^*\circ \overline{P})$ is given by $\{s\in \mathcal{D}(\overline{P}):\ \overline{P}s\in \mathcal{D}(\overline{P}^*)\}$. Consider now $\overline{P}^*\circ \overline{P}:L^2(M,E)\rightarrow L^2(M,E)$. Then, according to \cite{BL} pag. 98, we have  that 
$\bigcap_{k\in \mathbb{N}}\mathcal{D}((\overline{P}^*\circ \overline{P})^k)$ is dense in $\mathcal{D}(\overline{P}^*\circ \overline{P})$ with respect to the graph norm of $\overline{P}^*\circ \overline{P}$. By elliptic regularity we have that $\bigcap_{k\in \mathbb{N}}\mathcal{D}((\overline{P}^*\circ \overline{P})^k)$ $\subset C^{\infty}(M,E)$ and therefore $C^{\infty}(M,E)\cap \mathcal{D}(\overline{P}^*\circ \overline{P})$ is dense in $\mathcal{D}(\overline{P}^*\circ \overline{P})$ with respect to the graph norm of $\overline{P}^*\circ \overline{P}$. Thus, in order to complete the proof, we have to show that: 
\begin{itemize}
\item The inclusion $\mathcal{D}(\overline{P}^*\circ \overline{P})\hookrightarrow \mathcal{D}(\overline{P})$ is continuous where each space is endowed with the corresponding graph norm.
\item $\mathcal{D}(\overline{P}^*\circ \overline{P})$ is dense in $\mathcal{D}(\overline{P})$ with respect to the graph norm of $\overline{P}$.
\end{itemize}
The first point it is a consequence of the following inequality: for each $s\in \mathcal{D}(\overline{P}^*\circ \overline{P})$ $$\|\overline{P}s\|^2_{L^2(M,F)}=\langle s,\overline{P}^*(\overline{P}(s))\rangle_{L^2(M,E)}\leq \frac{1}{2}( \|s\|^2_{L^2(M,E)}+\|\overline{P}^*(\overline{P}(s))\|^2_{L^2(M,E)})$$ and therefore
$$\|s\|^2_{L^2(M,E)}+\|\overline{P}s\|^2_{L^2(M,E)}\leq \frac{3}{2} (\|s\|^2_{L^2(M,E)}+\|\overline{P}^*(\overline{P}(s))\|^2_{L^2(M,E)}).$$
For the second point we can argue in this way: let $v\in \mathcal{D}(\overline{P})$ and assume that for each $s\in \mathcal{D}(\overline{P}^*\circ \overline{P})$ we have $\langle v, s\rangle_{L^2(M,E)}+$ $\langle \overline{P}v,\overline{P}s\rangle_{L^2(M,F)}=0$. This is equivalent to say that $\langle v,s+\overline{P}^*(\overline{P}(s))\rangle_{L^2(M,E)}=0$. But $\id+(\overline{P}^*\circ \overline{P})$, where $\id$ is the identity operator,  has dense image because it is an injective and self-adjoint operator. We can therefore conclude that $v=0$ and this completes the proof.
\end{proof}

\begin{prop}
\label{carmina}
Let $(M,g)$,  $E$ and $F$ be as above.  Let  $ P:C^{\infty}_c(M,E)\rightarrow C^{\infty}_c(M,F)$ be  a first order differential operator. Let $s\in\mathcal{D}(P_{\max})$. Assume that there is  an open subset $U\subset M$  with compact closure such that $s|_{M\backslash \overline{U}}=0$. Then $s\in\mathcal{D}(P_{\min})$.

\end{prop}

\begin{proof}
The statement follows by   Lemma 2.1 in \cite{GL}.
\end{proof}

Now, in the remaining part of this section, we recall the notion of {\em Sobolev space associated to a metric connection.}
Consider again the bundle $E$ endowed with the metric $h$. Let $\nabla:C^{\infty}(M,E)\longrightarrow C^{\infty}(M,T^*M\otimes E)$ be a metric connection, that is a connection which satisfies the following property: for each $s, v\in C^{\infty}(M,E)$ we have $d(h(s,v))=h(\nabla s, v)+h(s, \nabla v)$. Clearly $h$ and $g$  induce in a natural way a metric on $T^*M\otimes E$ that we label by $\tilde{h}$. Let $\nabla^t:C^{\infty}_c(M,T^*M\otimes E)\longrightarrow C^{\infty}_c(M,E)$ be  the {\em formal adjoint} of $\nabla$ with respect to $\tilde{h}$ and $g$. Then  the Sobolev space $W^{1,2}(M,E)$ is  defined in the following way: 
\begin{align}
\label{cardo}
& W^{1,2}(M,E):=\{s\in L^2(M,E): \text{there is}\ v\in L^2(M,T^*M\otimes E)\ \text{such that}\ \int_{M}h(s,\nabla^t\phi)\dvol_g=\\ \nn&=\int_{M}\tilde{h}(v,\phi)\dvol_g\ \text{for each}\ \phi\in C^{\infty}_c(M,T^*M\otimes E)\}
\end{align}
Using \eqref{ner} we have $W^{1,2}(M,E)=\mathcal{D}(\nabla_{\max})$.
Moreover we have also   the Sobolev space $W^{1,2}_0(M,E)$ whose definition is the following:
\begin{align}
\label{scold}
& W^{1,2}_0(M,E):=\{s\in L^2(M,E)\ \text{such that there is a sequence}\ \{s_i\}\in C_c^{\infty}(M,E)\ \text{with}\ s_i\rightarrow s\\
\nn & \text{in}\ L^2(M,E)\ \text{and}\ \nabla s_i\rightarrow w\ \text{in}\ L^2(M,T^*M\otimes E)\ \text{to some }\ w\in L^2(M,T^*M\otimes E)\}.
\end{align}
Analogously to the previous case, using \eqref{spinaci}, we have $W^{1,2}_0(M,E)=\mathcal{D}(\nabla_{\min})$. As is this well known $W^{1,2}(M,E)$ and $W^{1,2}_0(M,E)$ are two Hilbert spaces. We adopt the same convention for the notation  we used before. Instead of writing $W^{1,2}(M,E,g,h)$ or $W^{1,2}_0(M,E,g,h)$ we will simply write $W^{1,2}(M,E)$ and $W^{1,2}_0(M,E)$. For the trivial bundle $M\times \mathbb{R}$ we will write $W^{1,2}(M,g)$ and $W^{1,2}_0(M,g)$. We   recall  the following result:
\begin{prop}
\label{lunedi}
Let $(M,g)$ be an open and possibly incomplete Riemannian manifold of dimension $m$. Let $E$ be a vector bundle over $M$ endowed with a metric $h$. Let $U\subset M$ be an open subset with compact closure. Consider the spaces $ L^2(U, E|_{U})$  and $W^{1,2}_0(U,E|_U)$ where $U$ is endowed with the metric $g|_U$.  Then the natural inclusion
\begin{equation}
\label{cccp}
W^{1,2}_0(U,E|_U)\hookrightarrow L^2(U, E|_{U})
\end{equation}
is a compact operator .  Therefore  the map $$i_0: W^{1,2}_0(U,E|_U)\rightarrow L^2(M, E)$$ given by
\begin{equation}
\label{amatasorella}
i_{0}(f)= \left\{
\begin{array}{ll}
f & on\ U\\
0 & on\ M\setminus U
\end{array}
\right.
\end{equation}
is an injective and  compact operator.
\end{prop}

\begin{proof}
See for example \cite{MM} pag. 349 or \cite{KWE} pag. 179 for \eqref{cccp}. Now \eqref{amatasorella} follows immediately by the following decomposition: $W^{1,2}_0(U,E|_U)\hookrightarrow L^2(U, E|_{U})\stackrel{i_0}{\rightarrow} L^2(M,E).$
\end{proof}

Consider now $d_k:\Omega^k_c(M)\rightarrow \Omega^{k+1}_c(M)$, the de Rham differential acting on the space of smooth $k$-forms with compact support. Given a Riemannian metric $g$ on $M$, we will label by $\langle\ ,\rangle_{g_k}$ and by $|\ |_{g_k}$ respectively the metric and the  pointwise norm induced by $g$ on $\Lambda^kT^*M$ for each $k=0,...,m$ where $m=\dim(M)$. In the case $k=1$, with a little abuse of notation, we will simply write $\langle\ ,\rangle_{g}$ and  $|\ |_{g}$ instead of $\langle\ ,\rangle_{g_1}$ and  $|\ |_{g_1}.$ Finally we will label by $\tilde{g}_k$ the metric that $g$ induces on $T^*M\otimes \Lambda^kT^*M$. Following \eqref{ner} and \eqref{spinaci} we will denote by $d_{k,\max/\min}:L^2\Omega^k(M,g)\rightarrow L^2\Omega^{k+1}(M,g)$ respectively the maximal and the minimal extension of $d_k$ acting on the space of $L^2$ $k$-forms. 

\begin{prop}
\label{pasta}
Let $(M,g)$ be an open and possibly incomplete Riemannian manifold of dimension $m$. Let $E$ be a vector bundle over $M$ endowed with a metric $h$. Let $s\in W^{1,2}(M,E)\cap C^{\infty}(M,E)$ and let $f\in \mathcal{D}(d_{0,\min})$ with compact support. Then $fs\in W^{1,2}_0(M,E)$ and we have $\nabla_{\min}(fs)=\eta \otimes s+f\nabla s$ where $\eta=d_{0,\min}f$.
\end{prop}

\begin{proof}
Let $U$ be an open subset of $M$ such that $\overline{U}$ is compact and $\supp(f)\subset U$. Let $\{\phi_j\}_{j\in \mathbb{N}}$ a sequences of smooth functions with compact support such that $\phi_j$ converges to $f$ in $\mathcal{D}(d_{0,\min})$ with respect to the graph norm. Consider now a smooth function with compact support $\gamma$ such that $\gamma|_{\supp(f)}=1$ and $\gamma|_{M\backslash U}=0$. Let $\{\psi_j\}_{j\in \mathbb{N}}$ be the sequence of smooth functions with compact support defined as $\psi_j=\gamma \phi_j$. Then it is immediate to check that  also $\psi_j$ converges to $f$ in $\mathcal{D}(d_{0,\min})$ with respect to the graph norm. Finally consider the sequence $\{\psi_js\}_{j\in \mathbb{N}}$.  We first note that $\eta \otimes s+f\nabla s\in L^2(M,T^*M\otimes E)$ because $$\|\eta \otimes s+f\nabla s\|_{ L^2(M,T^*M\otimes E)}\leq \|\eta\|_{L^2\Omega^1(M,g)}\|s|_U\|_{L^{\infty}(U,E)}+\|f\|_{L^{2}(M,g)}\|\nabla s|_U\|_{L^{\infty}(U,T^*U\otimes E)}.$$ Now in order to complete the proof we have to show that $\{\psi_js\}_{j\in \mathbb{N}}$  converges to $fs$ in the graph norm of $\nabla_{\min}$
We have $$\|\psi_js-fs\|^2_{L^2(M,E)}\leq \|s|_U\|^2_{L^{\infty}(U,E|_U)} \|\psi_j-f\|^2_{L^2(M,g)}$$ and therefore $\lim_{j\rightarrow \infty} \|\psi_js-fs\|^2_{L^2(M,E)}=0$. In the same way $$\|\psi_j\nabla s-f\nabla s\|^2_{L^2(M,T^*M\otimes E)}\leq \|\nabla s|_U\|^2_{L^{\infty}(U,T^*U\otimes E|_U)} \|\psi_j-f\|^2_{L^2(M,g)}$$ and therefore $\lim_{j\rightarrow \infty}\|\psi_j\nabla s-f\nabla s\|^2_{L^2(M,T^*M\otimes E)}=0$. Again $$\|d_0\psi_j \otimes s-\eta\otimes s\|^2_{L^2(M,T^*M\otimes E)}\leq \| s|_U\|^2_{L^{\infty}(U,E|_U)} \|d_0\psi_j-\eta\|^2_{L^2\Omega^1(M,g)}$$ and therefore $\lim_{j\rightarrow \infty}\|d_{0}\psi_j\otimes  s-\eta\otimes s\|^2_{L^2(M,T^*M\otimes E)}=0$. So we can conclude that $\psi_js$ converges to $fs$ in $\mathcal{D}(\nabla_{\min})$ with respect to the graph norm and that $\nabla_{\min}(fs)=\eta \otimes s+f\nabla s$.
\end{proof}

We conclude this section with the following proposition.
\begin{prop}
\label{corti}
Let $(M,g)$ be an open and  possibly incomplete Riemannian manifold. Let $E$ be a vector bundle over $M$, $h$ a metric on $E$ and let $\nabla:C^{\infty}(M,E)\rightarrow C^{\infty}(M,T^*M\otimes E)$ be a metric connection. Consider the Sobolev space $W^{1,2}(M,E)$. Then $C^{\infty}(M,E)\cap L^{\infty}(M,E)\cap W^{1,2}(M,E)$ is dense in $W^{1,2}(M,E)$.
\end{prop}

\begin{proof}
We give the proof in the case $E$ is a complex vector bundle endowed with a Hermitian metric $h$. When $E$ is real the proof is the same with the obvious modifications. According to Prop. \ref{partiro} it is enough to show that $C^{\infty}(M,E)\cap L^{\infty}(M,E)\cap W^{1,2}(M,E)$ is dense in $C^{\infty}(M,E)\cap W^{1,2}(M,E)$. Let $s\in C^{\infty}(M,E)\cap W^{1,2}(M,E)$ and define 
\begin{equation}
\label{venere}
s_n:= \frac{s}{(\frac{|s|_h^2}{n}+1)^{\frac{1}{2}}}.
\end{equation}
Clearly $s_n\in C^{\infty}(M,E)$. Moreover it is immediate to note that $(\frac{|s|_h^2}{n}+1)^{-\frac{1}{2}}\in L^{\infty}(M,g)$. Therefore, by the fact that $s\in L^2(M,E)$, we can conclude that $s_n\in L^2(M,E)\cap C^{\infty}(M,E)$. Finally we have $$|s_n|_h=\frac{|s|_h}{(\frac{|s|_h^2}{n}+1)^{\frac{1}{2}}}\leq n^{\frac{1}{2}}.$$  In this way we get  that $s_n \in C^{\infty}(M,E)\cap L^{\infty}(M,E)\cap L^{2}(M,E)$. Now consider $\nabla s_n$.  
We have: $$\nabla s_n= -\frac{1}{2}(\frac{|s|_h^2}{n}+1)^{-\frac{3}{2}}\frac{2}{n}\re(h(\nabla s, s))\otimes s+ (\frac{|s|^2_h}{n}+1)^{-\frac{1}{2}} \nabla s$$  
where $\re(h(\nabla s,s ))$ is the real part of $h(\nabla s,s )$. First of all we want to show that $\nabla s_n\in L^2(M,T^*M\otimes E)$. By the assumptions $\nabla s\in L^2(M,T^*M\otimes E)$. As remarked above  we have  $(\frac{|s|_h^2}{n}+1)^{-\frac{1}{2}}\in L^{\infty}(M,g)$. Therefore we can conclude that $(\frac{|s|^2_h}{n}+1)^{-\frac{1}{2}}\nabla s\in L^2(M,T^*M\otimes E).$ For $-\frac{1}{2}(\frac{|s|_h^2}{n}+1)^{-\frac{3}{2}}\frac{2}{n}\re(h(\nabla s, s))\otimes s$ we argue as follows. First of all we note that 
\begin{equation}
\label{zappa}
|-\frac{1}{2}(\frac{|s|_h^2}{n}+1)^{-\frac{3}{2}}\frac{2}{n}\re(h(\nabla s, s))\otimes s|_{\tilde{h}}\leq \frac{1}{n}(\frac{|s|_h^2}{n}+1)^{-\frac{3}{2}}|\nabla s|_{\tilde{h}}|s|_h^2.
\end{equation}
 It is clear that $(\frac{|s|_h^2}{n}+1)^{-\frac{3}{2}}|s|_h\in L^{\infty}(M,g)$ and  $|s|_h\in L^2(M, g)$. Therefore $(\frac{|s|_h^2}{n}+1)^{-\frac{3}{2}}|s|_h^2 \in  L^2(M, g)$. Moreover $(\frac{|s|_h^2}{n}+1)^{-\frac{3}{2}}|s|_h^2 \in  L^{\infty}(M, g)$. In fact we have 
\begin{equation}
\label{orianabella}
(\frac{|s|_h^2}{n}+1)^{-\frac{3}{2}}|s|^2_h\leq \frac{|s|^2_h}{(\frac{|s|_h^2}{n})+1}\leq n.
\end{equation}
 This shows that $(\frac{|s|_h^2}{n}+1)^{-\frac{3}{2}}|s|_h^2 \in  L^{\infty}(M, g)\cap L^2(M,g)$. By the fact that $|\nabla s|_{\tilde{h}} \in L^2(M,g)$ we can conclude that $$-\frac{1}{n}(\frac{|s|_h^2}{n}+1)^{-\frac{3}{2}}|\nabla s|_{\tilde{h}}|s|_h^2\in L^2(M,g).$$  According to \eqref{zappa} this implies that  $$-\frac{1}{n}(\frac{|s|_h^2}{n}+1)^{-\frac{3}{2}}\re(h(\nabla s, s))\otimes s\in  L^{2}(M,T^*M\otimes  E)$$ and in conclusion we proved that $\nabla s_n\in L^2(M,T^*M\otimes E)$. Finally the last step is to show that $s_n$ converges to $s$ in the graph norm of $\nabla$.
For $\|s-s_n\|^2_{L^2(M,E)}$ we have $$\|s-s_n\|^2_{L^2(M,E)}=\int_M (1-(\frac{|s|_h^2}{n}+1)^{-\frac{1}{2}})^2|s|_h^2\dvol_g$$ and  using the Lebesgue dominate convergence theorem we get $\lim_{n\rightarrow \infty} \|s-s_n\|^2_{L^2(M,E)}=0$. Now we show that $\lim_{n\rightarrow \infty} \|\nabla s-\nabla s_n\|^2_{L^2(M,T^*M\otimes E)}=0$. We have  $$\|\nabla s-\nabla s_n\|_{L^2(M,T^*M\otimes E)}\leq  \|-\frac{1}{n}(\frac{|s|_h^2}{n}+1)^{-\frac{3}{2}}\re(h(\nabla s, s))\otimes s\|_{L^2(M,T^*M\otimes E)}+ \|\nabla s-(\frac{|s|^2_h}{n}+1)^{-\frac{1}{2}}\nabla s\|_{L^2(M,T^*M\otimes E)}.$$ For  $\|\nabla s-(\frac{|s|^2_h}{n}+1)^{-\frac{1}{2}}\nabla s\|_{L^2(M,T^*M\otimes E)}$ we have  $$\|\nabla s-(\frac{|s|^2_h}{n}+1)^{-\frac{1}{2}}\nabla s\|^2_{L^2(M,T^*M\otimes E)}=\int_{M}(1-(\frac{|s|^2_h}{n}+1)^{-\frac{1}{2}})^2|\nabla s|^2_{\tilde{h}}\dvol_g$$ and,  again by the Lebesgue dominate convergence theorem, we can conclude that  $$\lim_{n\rightarrow \infty}\|\nabla s-(\frac{|s|^2_h}{n}+1)^{-\frac{1}{2}}\nabla s\|_{L^2(M,T^*M\otimes E)}=0.$$ For the remaining term, using \eqref{zappa},  we have 
\begin{align}
\nn& \|-\frac{1}{n}(\frac{|s|_h^2}{n}+1)^{-\frac{3}{2}}\re(h(\nabla s, s))\otimes s\|^2_{L^2(M,T^*M\otimes E)}=\int_M|-\frac{1}{n}(\frac{|s|_h^2}{n}+1)^{-\frac{3}{2}}\re(h(\nabla s, s))\otimes s|_{\tilde{h}}^2\dvol_g\leq\\ \nn&\leq\int_{M}\frac{1}{n^2}(\frac{|s|_h^2}{n}+1)^{-3}|\nabla s|^2_{\tilde{h}}|s|_h^4\dvol_g.
\end{align}
Using \eqref{orianabella} we have $$|s|_h^4(\frac{|s|_h^2}{n}+1)^{-3}=(|s|_h^2(\frac{|s|_h^2}{n}+1)^{-\frac{3}{2}})^2\leq n^2$$ and this in turn implies that $$\frac{1}{n^2}|\nabla s|_{\tilde{h}}^2|s|_h^4(\frac{|s|_h^2}{n}+1)^{-3}\leq |\nabla s|_{\tilde{h}}^2.$$
So we are again in  the position to apply the Lebesgue dominate convergence theorem and thus we can conclude that  $$0\leq \lim_{n\rightarrow \infty}\|-\frac{1}{n}(\frac{|s|_h^2}{n}+1)^{-\frac{3}{2}}\re(h(\nabla s, s))\otimes s\|_{L^2(M,T^*M\otimes E)}\leq \int_{M}\lim_{n\rightarrow \infty}\frac{1}{n^2}(\frac{|s|_h^2}{n}+1)^{-3}|\nabla s|^2_{\tilde{h}}|s|_h^4\dvol_g=0.$$ In conclusion we proved that  $s_n$ converges to $s$ in the graph norm of $\nabla$ and this complete the proof.
\end{proof}

\section{Kato's inequality and  domination of semigroups}
In this section we recall the Kato's inequality and then, following   the lines of \cite{HSUH} and \cite{HSU},  we discuss  its relation with the theory of domination of semigroups. Unlike \cite{HSUH} and  \cite{HSU} we are interested to apply this tool in an incomplete setting and this requires a more  careful analysis because in general the Laplacian $\Delta_0$,  with core domain given by the smooth functions with compact support, is not longer an essentially self-adjoint operator.

\subsection{Kato's inequality}

Consider again an open and possibly incomplete Riemaniann manifold $(M,g)$. Let $E$ be a vector bundle over $M$ and let $h$ be a metric on $E$, Hermitian whether  $E$ is complex , Riemannian whether $E$ is  real. Finally let $\tilde{h}$ be the natural metric that $h$ and $g$ induce on $T^*M\otimes E$.
\begin{prop}
\label{devo}
Let $M$, $g$, $E$, $h$ and $\tilde{h}$ be as described above. Let $\nabla:C^{\infty}(M,E)\longrightarrow C^{\infty}(M,T^*M\otimes E)$ be a metric connection and let $s\in C^{\infty}(M,E)$. Let $Z\subset M$ be the zero set of $s$. Then  on $M\setminus Z$ we have the following pointwise inequality: 
\begin{equation}
\label{kato}
|d(|s|_h)|_g\leq |\nabla s|_{\tilde{h}}.
\end{equation}
If $s\in W^{1,2}(M,E)\cap C^{\infty}(M,E)$ then $|s|_h\in \mathcal{D}(d_{0,\max})$ and we have 
\begin{equation}
\label{tratti}
\|d_{0,\max}|s|_h\|_{L^2\Omega^1(M,g)}\leq \|\nabla s\|_{L^2(M,T^*M\otimes E)}
\end{equation}
In particular, if $(E,h)$ is a complex vector bundle endowed with a Hermitian metric,  $d_{0,\max}(|s|_h)$ satisfies:
\begin{equation}
\label{cucchi}
d_{0,\max}(|s|_h)= \left\{
\begin{array}{ll}
\re(h(\nabla s,s))|s|^{-1}_h\ & on\ M\setminus Z\\
0 & on\ Z
\end{array}
\right.
\end{equation}
while if $(E,h)$ is a real vector bundle endowed with a Riemannian  metric,  $d_{0,\max}(|s|_h)$ satisfies:
\begin{equation}
\label{cucchiai}
d_{0,\max}(|s|_h)= \left\{
\begin{array}{ll}
h(\nabla s,s)|s|^{-1}_h\ & on\ M\setminus Z\\
0 & on\ Z
\end{array}
\right.
\end{equation}
\end{prop}

\begin{proof}
As for the previous proof we treat only  the complex case. The proof for the real case is completely analogous with the obvious modifications. The proof of \eqref{kato} is based on the following observations: On $M\setminus Z$ we have $|d(|s|^2_h)|_g=2|s|_h|d(|s|_h)|_g$. On the other hand $|d(|s|^2_h)|_g=2|\re(h(\nabla s, s))|_g\leq 2|\nabla s |_{\tilde{h}}|s|_h.$ Therefore  \eqref{kato} holds. For \eqref{tratti} and \eqref{cucchi} we argue as in \cite{PB} VI.31. Consider $\phi\in \Omega^1_c(M)$ and let $\epsilon_n:=\frac{1}{n}$. Then 
\begin{align}
\nn &\int_{M}|s|_h(d^t_0\phi) \dvol_g=\lim_{n\rightarrow \infty}\int_{M}(|s|_h^2+\epsilon_n^2)^{\frac{1}{2}}(d^t_0\phi) \dvol_g=\lim_{n\rightarrow \infty}\int_{M}\langle d_0((|s|_h^2+\epsilon_n^2)^{\frac{1}{2}}),\phi\rangle_g \dvol_g\\
 \nn &=\lim_{n\rightarrow \infty}\int_{M}\langle (|s|_h^2+\epsilon^2_n)^{-\frac{1}{2}}\re(h(\nabla s, s)),\phi\rangle_g \dvol_g=\int_{M}\langle \eta,\phi\rangle_g\dvol_g
\end{align} where $\eta$ is defined as in \eqref{cucchi}. In particular  for the pointwise norms we have  $|\eta|_g\leq |\nabla s|_{\tilde{h}}$. Therefore,  if  $s\in W^{1,2}(M,E)\cap C^{\infty}(M,E)$, we can conclude that $|s|_h\in \mathcal{D}(d_{0,\max})$ and that \eqref{tratti} and \eqref{cucchi} hold. 
\end{proof}

\begin{cor}
\label{katocor}
Under the assumptions of Prop. \ref{devo}.
\begin{itemize}
\item If $s\in W^{1,2}(M,E)$ then $|s|_h\in \mathcal{D}(d_{0,\max})$ and we have $\|d_{0,\max}|s|_h\|_{L^2\Omega^1(M,g)}\leq \|\nabla_{\max} s\|_{L^2(M,T^*M\otimes E)}$,
\item If $s\in W_0^{1,2}(M,E)$ then $|s|_h\in  \mathcal{D}(d_{0,\min})$ and we have $\|d_{0,\min}|s|_h\|_{L^2\Omega^1(M,g)}\leq \|\nabla_{\min} s\|_{L^2(M,T^*M\otimes E)}$.
\end{itemize}
\end{cor}
\begin{proof}
Let $s\in W^{1,2}(M,E)$. According to Prop. \ref{corti} there is a sequence $\{\phi_n\}_{n\in \mathbb{N}}\subset W^{1,2}(M,E)\cap C^{\infty}(M,E)$   such that $\lim_{n\rightarrow \infty} \phi_n=s$ in $W^{1,2}(M,E)$. We have $\lim_{n\rightarrow \infty} |\phi_n|_h=|s|_h$ in $L^{2}(M,g)$ and, using  \eqref{tratti} and the fact that $\nabla\phi_n\rightarrow \nabla_{\max} s$ in $L^2(M,T^*M\otimes E)$, we get $\|d_{0,\max}|\phi_n|_h\|_{L^2\Omega^1(M,g)}\leq \tau$ for some positive $\tau\in \mathbb{R}.$ This implies  the existence of a subsequence $\{\phi_n'\}_{n\in \mathbb{N}}\subset \{\phi_n\}_{n\in \mathbb{N}}$ such that $\{d_{0,\max}(|\phi'_n|_h)\}_{n\in \mathbb{N}}$ converges weakly in $L^2\Omega^1(M,g)$ to some $\beta\in L^2\Omega^1(M,g)$, see for instance \cite{JBC} pag. 132.  Now let $\omega\in \Omega^1_c(M)$. We have $$\langle |s|_h, d^t_{0}\omega \rangle_{L^2(M,g)}=\lim_{n\rightarrow \infty}\langle |\phi'_n|_h, d^t_{0}\omega \rangle_{L^2(M,g)}=\lim_{n\rightarrow \infty}\langle d_{0,\max}|\phi'_n|_h, \omega \rangle_{L^2\Omega^1(M,g)}=\langle \beta, \omega \rangle_{L^2\Omega^1(M,g)}.$$ Therefore, according to \eqref{ner}, we proved that  $|s|_h\in \mathcal{D}(d_{0,\max})$ and $d_{0,\max}|s|_h=\beta$.\\  Now to estimate $\|d_{0,\max}|s|_h\|_{L^2\Omega^1(M,g)}$, using \eqref{tratti}, we have 
\begin{align}
\nn & \|d_{0,\max}|s|_h\|^2_{L^2\Omega^1(M,g)}=\|\beta\|^2_{L^2\Omega^1(M,g)}=\lim_{n\rightarrow \infty}\langle  d_{0,\max}|\phi'_n|_h,\beta \rangle_{L^2\Omega^1(M,g)}\leq \lim_{n\rightarrow \infty} \| d_{0,\max}|\phi'_n|_h\|_{L^2\Omega^1(M,g)}\|\beta \|_{L^2\Omega^1(M,g)}\\
\nn & \leq \lim_{n\rightarrow \infty} \| \nabla\phi'_n\|_{L^2(M,T^*M\otimes E)}\|\beta \|_{L^2\Omega^1(M,g)}=\| \nabla_{\max} s\|_{L^2(M,T^*M\otimes E)}\|\beta \|_{L^2\Omega^1(M,g)}.\ \text{Hence the first point is proved.}
\end{align}
For the second point we argue in a similar manner. Let $s\in W^{1,2}_0(M,E)$ and let $\{\psi_n\}_{n\in \mathbb{N}}$ be a sequence of smooth sections with compact support such that $\lim_{n\rightarrow \infty} \psi_n=s$ in $W^{1,2}_0(M,E)$. As in the previous case we  have $\lim_{n\rightarrow \infty} |\psi_n|_h=|s|_h$ in $L^{2}(M,g)$ and, using  \eqref{tratti} and the fact that $\nabla\psi_n\rightarrow \nabla_{\min}s$ in $L^2(M,T^*M\otimes E)$, we get $\|d_{0,\max}|\psi_n|_h\|_{L^2\Omega^1(M,g)}\leq \tau'$ for some positive $\tau'\in \mathbb{R}.$ This implies  the existence of a subsequence $\{\psi_n'\}_{n\in \mathbb{N}}\subset \{\psi_n\}_{n\in \mathbb{N}}$ such that $\{d_{0,\max}(|\psi'_n|_h)\}_{n\in \mathbb{N}}$ converges weakly in $L^2\Omega^1(M,g)$ to some $\gamma\in L^2\Omega^1(M,g)$. Moreover we observe that $|\psi_n|_h\in \mathcal{D}(d_{0,\min})$ because $|\psi_n|_h\in \mathcal{D}(d_{0,\max})$ and it has  compact support, see Prop. \ref{carmina}. Now let $\omega\in \mathcal{D}(d^t_{0,\max})$. We have $$\langle |s|_h, d^t_{0,\max}\omega \rangle_{L^2(M,g)}=\lim_{n\rightarrow \infty}\langle |\psi'_n|_h, d^t_{0,\max}\omega \rangle_{L^2(M,g)}=\lim_{n\rightarrow \infty}\langle d_{0,\min}|\psi'_n|_h, \omega \rangle_{L^2\Omega^1(M,g)}=\langle \gamma, \omega \rangle_{L^2\Omega^1(M,g)}.$$ Therefore, for each $\omega\in \mathcal{D}(d^t_{0,\max})$, we have $\langle |s|_h, d^t_{0,\max}\omega \rangle_{L^2(M,g)}\leq \|\gamma\|_{L^2\Omega^1(M,g)}\|\omega\|_{L^2\Omega^1(M,g)}$ and this shows that $|s|_h\in \mathcal{D}((d^t_{0,\max})^*)$ that is $|s|_h\in \mathcal{D}(d_{0,\min})$. Finally, by the previous point, we have  $$\|d_{0,\min}(|s|_h)\|_{L^2\Omega^1(M,g)}=\|d_{0,\max}(|s|_h)\|_{L^2\Omega^1(M,g)}\leq \|\nabla_{\max} s\|_{L^2(M,T^*M\otimes E)}=\|\nabla_{\min} s\|_{L^2(M,T^*M\otimes E)}$$ and the proposition is thus established.
\end{proof}

\subsection{A brief reminder on quadratic forms and the Friedrich extension}
In this subsection we give a very brief account about some results on  quadratic forms and the Friedrich extension of a positive and symmetric operator. We follow  the Appendix C.1 in \cite{MM} and we refer to it for the proofs. For a thorough treatment of the subject  we refer to \cite{RSI} and \cite{RSII}.
Let $H$ be a   Hilbert space with inner product $\langle\ ,\ \rangle$. Let $B:H\rightarrow H$ be a linear, unbounded and densely defined operator. Assume that $B$ is   symmetric and positive, that is $B$ is   extended by its adjoint $B^*$ and $\langle Bu,u \rangle\geq 0$ for each $u\in \mathcal{D}(B)$. The quadratic form associated to $B$, usually labeled by $Q_B$, is by definition $Q_B(u,v):=\langle Bu,v \rangle.$ Let $\langle\ , \rangle_{B}$ be  the inner product given by $\langle\ , \rangle+Q_B$ and let  $\mathcal{D}(Q_B)$ be the completion of $\mathcal{D}(B)$ through $\langle\ , \rangle_B$. It is immediate to check that the identity $\id:\mathcal{D}(B)\rightarrow \mathcal{D}(B)$ extends as a bounded and injective map $i_{Q_B}:\mathcal{D}(Q_B)\rightarrow H$. Therefore in what follows we will identify $\mathcal{D}(Q_B)$ with its image in $H$ through $i_{Q_B}$ which is given by $$\{u\in H:\ \text{there exists}\ \{u_n\}_{n\in \mathbb{N}}\subset \mathcal{D}(B)\ \text{such that}\ \langle u_n-u,u_n-u \rangle\rightarrow 0\ \text{and}\  \langle u_n-u_m,u_n-u_m \rangle_B\rightarrow 0\ \text{as}\ m,n\rightarrow \infty\}.$$
Now we define the {\em Friedrich extension} of $B$, labeled by $B^{\mathcal{F}}$, as the operator whose {\em domain} is given by $$\{u\in \mathcal{D}(Q_B):\ \text{there exists}\ v\in H\ \text{with}\ Q_B(u,w)=\langle v,w\rangle\ \text{for any}\ w\in \mathcal{D}(Q_B)\}$$ and we put $B^{\mathcal{F}}u:=v$. Defined in this way $B^{\mathcal{F}}$ is a {\em positive} and {\em self-adjoint} operator.  Moreover the above construction is equivalent to require that $\mathcal{D}(B^{\mathcal{F}})=\{u\in \mathcal{D}(B^*):\ \text{there exists}\ \{u_{n}\}\subset \mathcal{D}(B)\  \text{such that}\  \langle u-u_n,u-u_n\rangle \rightarrow 0\ \text{and}\  \langle B(u_n-u_m),u_n-u_m\rangle \rightarrow 0\ \text{for}\ n,m\rightarrow \infty\}$ and   $B^{\mathcal{F}}(u)=B^*(u)$, that is in a concise way $$\mathcal{D}(B^{\mathcal{F}}):=\mathcal{D}(Q_B)\cap \mathcal{D}(B^*)\ \text{and}\ B^{\mathcal{F}}u:=B^*u$$ for $u\in \mathcal{D}(B^{\mathcal{F}})$.
We conclude this reminder with the following result:
\begin{prop}
\label{fall}
 Let $E,F$ be two vector bundles over an open and  possibly incomplete  Riemannian manifold $(M,g).$ Let $h_E$ and $h_F$ be two metrics on $E$ and $F$ respectively. Let $D:C^{\infty}_{c}(M,E)\rightarrow C^{\infty}_{c}(M,F)$ be an unbounded and densely defined differential operator. Let $D^t:C^{\infty}_{c}(M,F)\rightarrow C^{\infty}_{c}(M,E)$ be its formal adjoint. Then for  $D^t\circ D:L^{2}(M,E)\rightarrow L^{2}(M,E)$ we have: 
\begin{enumerate}
\item $(D^t\circ D)^{\mathcal{F}}=(D^t)_{\max}\circ D_{\min}.$
\item $\mathcal{D}(Q_{D^t\circ D})=\mathcal{D}(D_{\min})$ and $Q_{D^t\circ D}(u,v)=\langle D_{\min}u, D_{\min}v\rangle_{L^2(M,E)}$ for each $u,v\in \mathcal{D}(Q_{D^t\circ D}).$
\end{enumerate}
\end{prop}

\begin{proof}
Both statements follow immediately from the definitions  and the constructions given above. Moreover the first point is also proved in  \cite{BLE}, pag. 447.
\end{proof}

\subsection{Domination of semigroups}

We refer to \cite{BGV} and to \cite{GYA}  for the background on the heat operator.

\begin{teo}
\label{ready}
Let $(M,g)$ be an open and  possibly incomplete Riemannian manifold. Let $E$ be a  vector bundle on $M$ and let $h$ be a metric on $E$. Let $\nabla$ be a metric connection, let $\nabla^t$ be  its formal adjoint  and let  
\begin{equation}
\label{gazzo}
P:C^{\infty}_c(M,E)\rightarrow C^{\infty}_c(M,E),\ P= (\nabla^t\circ \nabla)+ L
\end{equation} 
be a Schr\"odinger type operator with $L\in \End(E)$ such that 
\begin{itemize}
\item $P$ is symmetric and  positive.
\item  There is a constant $c\in \mathbb{R}$ such that, for each $s\in C^{\infty}(M,E)$, we have  $$h(Ls,s)\geq ch(s,s).$$
\end{itemize}
Let $P^{\mathcal{F}}$be  the Friedrich extension of $P$ and let $\Delta_{0}^{\mathcal{F}}$ be the Friedrich extension of the  Laplacian acting on smooth functions with compact support $\Delta_{0}:C^{\infty}_c(M)\rightarrow C^{\infty}_c(M)$. Then, for  the respective heat operators $e^{-tP^{\mathcal{F}}}$ and $e^{-t\Delta_0^{\mathcal{F}}}$,  we have the following  domination of semigroups:
\begin{equation}
\label{vis}
|e^{-tP^{F}}s|_h\leq e^{-tc}e^{-\Delta_0^{\mathcal{F}}}|s|_h
\end{equation}
 for each $s\in L^2(M,E).$
\end{teo}
This theorem  is proved in \cite{Batu},  Theorem 2.13. Here we provide  a different proof, under some additional assumptions,  in the spirit of \cite{HSUH}.  Our additional assumptions are:
\begin{itemize}
\item $\mathcal{D}(d_{0,\max})=\mathcal{D}(d_{0,\min})$ on $(M,g)$.
\item $\vol_g(M)<\infty$.
\end{itemize}
Clearly the second assumption is satisfied in our cases of interest, that is  $M$ is the regular part of a complex projective variety $V\subset \mathbb{C}\mathbb{P}^n$ and $g$ is the K\"ahler metric induced by the Fubini-Study metric on $\mathbb{C}\mathbb{P}^n$ or $M$ is the regular part of a smoothly Thom-Mather stratified pseudomanifold endowed with an iterated edge metric. Moreover, as we will see in Prop. \ref{taglia} and in Prop. \ref{tagliaz} , also the first assumption is fulfilled in our cases of interest.
We give the proof assuming that  $E$ is a Hermitian vector bundle.  In the real case the proof is analogous with the obvious modifications. We divide the proof  through several propositions. In order to state the first result  we recall from  \cite{RSIV} pag. 201 the following notion: Let $(X,\mu)$ be a $\sigma$-finite measure space. A function $f\in L^2(X,\mu)$ is called positive if it is non negative almost everywhere and  is not the zero function. A bounded operator $A:L^2(X,\mu)\rightarrow L^2(X,\mu)$ is called positivity preserving if $Af$ is positive whenever $f$ is positive.
\begin{prop}
\label{primo}
Let $(M,g)$ be an open and possibly incomplete Riemannian manifold.  Let $\Delta_{0}^{\mathcal{F}}$ be the Friedrich extension of the  Laplacian acting on smooth functions with compact support $\Delta_{0}:C^{\infty}_c(M)\rightarrow C^{\infty}_c(M)$. Consider the heat operator $e^{-t\Delta_0^{\mathcal{F}}}:L^{2}(M,g)\longrightarrow L^2(M,g)$. Then 
$e^{-t\Delta_0^{\mathcal{F}}}$ is positivity preserving for all $t>0$. 
\end{prop}

\begin{proof}
According to the Beurling-Deny criterion, see \cite{RSIV} pag. 209 or Theorem 3 in the appendix of \cite{PB}, the statement is equivalent to prove that if  $f\in \mathcal{D}(Q_{\Delta_0^{\mathcal{F}}})$ then $|f|\in \mathcal{D}(Q_{\Delta_0^{\mathcal{F}}})$ and  $Q_{\Delta_0^{\mathcal{F}}}(|f|,|f|)\leq Q_{\Delta_0^{\mathcal{F}}}(f,f)$. By Prop. \ref{fall} this condition becomes: 
for each $f\in \mathcal{D}(d_{0,\min})$ we have $|f|\in \mathcal{D}(d_{0,\min})$ and $$\langle d_{0,\min}|f|,d_{0,\min}|f|\rangle_{L^2\Omega^1(M,g)}\leq \langle d_{0,\min}f,d_{0,\min}f\rangle_{L^2\Omega^1(M,g)}.$$ Finally this last inequality is a consequence of  Prop. \ref{devo} and Cor. \ref{katocor}.
\end{proof}

\begin{prop}
\label{step}
Under the assumption of Prop. \ref{primo}. For each $\lambda>0$ the operator $$(\Delta_{0}^{\mathcal{F}}+\lambda)^{-1}:L^2(M,g)\rightarrow L^2(M,g)$$ is positivity preserving.
\end{prop}

\begin{proof}
It is a consequence of the following formula combined with Prop. \ref{primo}: $$(\Delta_{0}^{\mathcal{F}}+\lambda)^{-1}=\int_{0}^{\infty}e^{-\lambda t}e^{-t\Delta_{0}^{\mathcal{F}}}dt,\ \lambda>0.$$ See for instance \cite{PB}, Prop. 2 in the appendix.
\end{proof}

\begin{prop}
\label{gaugamelazz}
Let $(M,g)$ be an open and possibly incomplete Riemannian manifold. Let $P:C^{\infty}_c(M,E)\rightarrow C_c^{\infty}(M,E)$ be as in the statement of  Theorem \ref{ready}. Given $s\in C^{\infty}(M,E)$ and $\epsilon >0$ let us  define $|s|_{h,\epsilon}$ as  $|s|_{h,\epsilon}:=(|s|^2_h+\epsilon^2)^{\frac{1}{2}}$. 
\begin{itemize}
\item Assume that $\vol_g(M)<\infty$. Let $s\in \mathcal{D}(P_{\max})\cap  C^{\infty}(M,E)$.  Then $|s|_{h,\epsilon}\in \mathcal{D}(\Delta_{0,\max})\cap  C^{\infty}(M)$. 
\item Assume that  $\vol_g(M)<\infty$ and that $\mathcal{D}(d_{0,\min})=\mathcal{D}(d_{0,\max})$. Let $s\in \mathcal{D}(P^{\mathcal{F}})\cap C^{\infty}(M,E)$. Then\\ $|s|_{h,\epsilon}\in \mathcal{D}(\Delta_0^{\mathcal{F}})\cap C^{\infty}(M)$.
\end{itemize}
\end{prop}
\begin{proof}
According to \cite{HSU}  pag. 30-31, we have $$\re(h(Ps,s))=\re (h((\nabla^t\circ \nabla)s+Ls,s))=\re (h((\nabla^t\circ \nabla)s,s))+h(Ls,s)\geq |s|_{h,\epsilon}\Delta_0|s|_{h,\epsilon}+c|s|^2_{h}.$$ Therefore $$|\Delta_0|s|_{h,\epsilon}|\leq |Ps|_h\frac{|s|_h}{|s|_{h,\epsilon}}+|c||s|_h\frac{|s|_h}{|s|_{h,\epsilon}}.$$ In this way we can conclude that $|\Delta_0|s|_{h,\epsilon}|\in L^2(M,g)$ because $|Ps|_h\in L^2(M,g)$, $|s|_h\in L^2(M,g)$ and $\frac{|s|_h}{|s|_{h,\epsilon}}\in  L^{\infty}(M,g)$. By the fact that $\vol_g(M)<\infty$ we have that $|s|_{h,\epsilon}\in L^2(M,g)$ and thus the proof of the first point is complete.  For the second point we can argue in this way: let $m\geq |c|+1$. Then, for each $\phi\in C^{\infty}_c(M,E)$, we have 
\begin{align}
\nn &m(\langle \phi,\phi\rangle_{L^2(M,E)}+\langle P\phi,\phi\rangle_{L^2(M,E)})\geq m\langle \phi,\phi\rangle_{L^2(M,E)}+\langle P\phi,\phi\rangle_{L^2(M,E)}=\\
 \nn &=m\langle \phi,\phi\rangle_{L^2(M,E)}+\langle (\nabla^t\circ\nabla)\phi,\phi\rangle_{L^2(M,E)}+\langle L\phi,\phi\rangle_{L^2(M,E)}\geq\\
\nn & \geq m\langle \phi,\phi\rangle_{L^2(M,E)}+\langle (\nabla^t\circ\nabla)\phi,\phi\rangle_{L^2(M,E)}+c\langle \phi,\phi\rangle_{L^2(M,E)}\geq \langle (\nabla^t\circ\nabla)\phi,\phi\rangle_{L^2(M,E)}+\langle \phi,\phi\rangle_{L^2(M,E)}.
\end{align}
 Therefore, if $Q_{P}$ is the quadratic form associated to $P$ and $Q_{\nabla^t\circ \nabla}$ is the quadratic form associated to $\nabla^t\circ \nabla$, we proved that 
\begin{equation}
\label{voxz}
\langle\ , \rangle_{L^2(M,E)}+Q_{P}\geq \frac{1}{m}(\langle\ , \rangle_{L^2(M,E)}+Q_{\nabla^t\circ \nabla}).
\end{equation}
This implies immediately that the identity $\id: C^{\infty}_c(M,E)\rightarrow C^{\infty}_c(M,E)$ induces a linear, bounded  and injective map $i:\mathcal{D}(Q_{P})\rightarrow \mathcal{D}(Q_{\nabla^t\circ \nabla}).$ Now if we take $s\in \mathcal{D}(P^{\mathcal{F}})\cap C^{\infty}(M,E)$ we know that $s\in \mathcal{D}(Q_{P})\cap \mathcal{D}(P_{\max})$. By the fact that $s\in \mathcal{D}(P_{\max})$, as a consequence of  the first point of this  proposition, we get $|s|_{h,\epsilon}\in \mathcal{D}(\Delta_{0,\max})$. By the fact that  $s\in \mathcal{D}(Q_{P})$, using \eqref{voxz} and Prop. \ref{fall}, we get $$s\in \mathcal{D}(\nabla_{\min})\ \text{and}\ Q_{\nabla^t\circ \nabla}(s,s)=\langle\nabla_{\min}s,\nabla_{\min}s\rangle_{L^2(M,T^*M\otimes E)}.$$ Now, using Kato's inequality in Prop. \ref{devo} and the fact that $\vol_g(M)<\infty$, we have $|s|_{h,\epsilon}\in \mathcal{D}(d_{0,\max})$. Finally the assumption  $\mathcal{D}(d_{0,\max})=\mathcal{D}(d_{0,\min})$ implies that  $|s|_{h,\epsilon}\in \mathcal{D}(d_{0,\min})$. Therefore  $|s|_{h,\epsilon}\in \mathcal{D}(d_{0,\min})\cap \mathcal{D}(\Delta_{0,\max})$. This in turn implies immediately  that $|s|_{h,\epsilon}\in \mathcal{D}(d^t_{0,\max}\circ d_{0,\min})$ that is $|s|_{h,\epsilon}\in \mathcal{D}(\Delta_{0}^{\mathcal{F}})$ according to Prop. \ref{fall}. 
\end{proof}

\begin{prop}
\label{disu}
Under the assumptions of Theorem \ref{ready}. Let $\mathcal{D}_0$ be defined as $\mathcal{D}_0:=\mathcal{D}(P^{\mathcal{F}})\cap C^{\infty}(M,E)$. Then for all $\lambda>0$, all $s\in \mathcal{D}_0$, all $f\in  C_c^{\infty}(M)$ such that  $f\geq 0$, there exists $u\in L^2(M,E)$ such that
\begin{enumerate}
\item $|u|_{h}=(\Delta_0^{\mathcal{F}}+\lambda)^{-1}f$.
\item $\langle s,u\rangle_{L^2(M,E)}=\langle |s|_h, |u|_h\rangle_{L^2(M,g)}.$
\item $\re\langle P^{\mathcal{F}} s,u\rangle_{L^2(M,E)}\geq \langle |s|_h,(\Delta_0^{\mathcal{F}}+c)|u|_h\rangle_{L^2(M,g)}$
\end{enumerate}
\end{prop}

\begin{proof}
For the first two points we follow the construction given in \cite{HSUH} pag. 895. Let $g=(\Delta_0^{\mathcal{F}}+\lambda)^{-1}f$. Let $e$ be a measurable section of $E$ such that  $h_p(e(p),e(p))=1$ for each $p\in M$. Define $\si (s)$ as 
\begin{equation}
\label{cassiusclayx}
\si (s):=\left\{
\begin{array}{ll}
\frac{s}{|s|_h} & s(p)\neq 0\\
e & s(p)=0
\end{array}
\right.
\end{equation}
Now define $u:= g\si (s)$. It follows immediately that  $\langle s,u\rangle_{L^2(M,E)}=\langle |s|_h, |u|_h\rangle_{L^2(M,g)}$ and that $|u|_h=g$. In particular the last equality follows by Prop. \ref{step}. This proves the first two points of the proposition. About the  third point, by the fact that $s$ is smooth, we will write simply $Ps$ instead of $P^{\mathcal{F}}s$. As explained in the proof of Prop. \ref{gaugamelazz}, we have  $\re(h(Ps,s))\geq |s|_{h,\epsilon}\Delta_0|s|_{h,\epsilon}+c|s|^2_h$, that is $\re(h(Ps,s/|s|_{h,\epsilon}))\geq\Delta_0|s|_{h,\epsilon}+c|s|^2_h/|s|_{h,\epsilon}$. This implies that $$\re(h(Ps,|u|_h\frac{s}{|s|_{h,\epsilon}}))\geq(\Delta_0|s|_{h,\epsilon})|u|_h+c\frac{|s|^2_h}{|s|_{h,\epsilon}}|u|_h.$$ We can integrate because $Ps\in L^2(M,E)$, $|s|_h$, $|u|_h\in L^2(M,g)$, $s/|s|_{h,\epsilon}\in L^{\infty}(M,E)$ and $\Delta_0|s|_{h,\epsilon}\in L^2(M,g).$ In this way we get $$\int_M\re(h(Ps,|u|_h\frac{s}{|s|_{h,\epsilon}}))\dvol_g\geq \int_M(\Delta_0|s|_{h,\epsilon})|u|_h\dvol_g+\int_Mc\frac{|s|^2_h}{|s|_{h,\epsilon}}|u|_h\dvol_g$$ that is 
\begin{equation}
\label{sonnya}
\re \langle Ps, |u|_h\frac{s}{|s|_{h,\epsilon}}\rangle_{L^2(M,E)}\geq \langle \Delta_0|s|_{h,\epsilon}, |u|_h\rangle_{L^2(M,g)}+\langle c\frac{|s|^2_h}{|s|_{h,\epsilon}}, |u|_h \rangle_{L^2(M,g)}.
\end{equation}
For the right hand side of \eqref{sonnya} we know that $|u|_h\in \mathcal{D}(\Delta_0^{\mathcal{F}})$ because $|u|_{h}=(\Delta_0^{\mathcal{F}}+\lambda)^{-1}f$ and $\lambda>0$. Moreover, by Prop. \ref{gaugamelazz}, we also know  that $|s|_{h,\epsilon}\in \mathcal{D}(\Delta_0^{\mathcal{F}})$.  Therefore, on the right hand side of \eqref{sonnya}, integration by part is allowed. This lead us to the expression $$\re \langle Ps, |u|_h\frac{s}{|s|_{h,\epsilon}}\rangle_{L^2(M,E)}\geq \langle |s|_{h,\epsilon}, \Delta_0^{\mathcal{F}}|u|_h \rangle_{L^2(M,g)}+\langle \frac{|s|^2_h}{|s|_{h,\epsilon}}, c|u|_h \rangle_{L^2(M,g)}$$ that is 
\begin{equation}
\label{listonx}
\int_{M}\re(h(Ps,|u|_h\frac{s}{|s|_{h,\epsilon}}))\dvol_g\geq \int_{M}(|s|_{h,\epsilon}\Delta_0^{\mathcal{F}}|u|_h+c\frac{|s|^2_h}{|s|_{h,\epsilon}}|u|_h)\dvol_g.
\end{equation}
Keeping in mind \eqref{cassiusclayx} and applying the Lebesgue's dominate convergence theorem, \eqref{listonx} becomes
$$\int_{M}\re(h(Ps,u))\dvol_g\geq \int_{M}|s|_{h}(\Delta_0^{\mathcal{F}}|u|_h+c|u|_h)\dvol_g$$ that is $$\re\langle Ps,u\rangle_{L^2(M,E)}\geq \langle |s|_{h}, (\Delta_0^{\mathcal{F}}+c)|u|_h\rangle_{L^2(M,g)}.$$ 
\end{proof}

Finally we have the last proposition.

\begin{prop}
\label{ulimo}
Under the assumptions of Prop. \ref{disu}. For each $\mu>\max\{0,-c\}$ and for each $\beta\in L^2(M,E)$ we have
\begin{equation}
\label{mosul}
(\Delta_0^\mathcal{F}+c+\mu)^{-1} |\beta|_h\geq  |(P^\mathcal{F}+\mu)^{-1} \beta|_h
\end{equation}
This in turn implies that 
\begin{equation}
\label{mosulz}
e^{-t(\Delta_0^\mathcal{F}+c)} |\beta|_h\geq  |e^{-tP^\mathcal{F}} \beta|_h
\end{equation}
and therefore
\begin{equation}
\label{mosulzx}
e^{-tc}e^{-t\Delta_0^\mathcal{F}} |\beta|_h\geq  |e^{-tP^\mathcal{F}} \beta|_h
\end{equation}
\end{prop}

\begin{proof}
Let $\mathcal{D}_0=\mathcal{D}(P^{\mathcal{F}})\cap C^{\infty}(M,E)$, $s\in \mathcal{D}_0$, $f\in C^{\infty}_c(M)$, $f\geq 0$, $\mu>-c$ and $u\in L^2(M,E)$ such that $|u|_h=(\Delta_0^{\mathcal{F}}+\mu+c)^{-1}f$.
Then, by Prop. \ref{disu}, we know that $$\re\langle P^{\mathcal{F}} s,u\rangle_{L^2(M,g)}\geq \langle |s|_h,(\Delta_0^{\mathcal{F}}+c)|u|_h\rangle_{L^2(M,g)}.$$ By the second point of Prop. \ref{disu}, for each $\mu\geq 0$, we still  have $$\re\langle (\mu+P^{\mathcal{F}}) s,u\rangle_{L^2(M,E)}\geq \langle |s|_h,(\Delta_0^{\mathcal{F}}+c+\mu)|u|_h\rangle_{L^2(M,g)}.$$ The previous inequality, requiring $\mu> \max\{0,-c\}$,  produces

\begin{equation}
\label{pfss}
\langle (|\mu+P^{\mathcal{F}}) s|_h,|u|_h\rangle_{L^2(M,g)}\geq \langle |s|_h, f\rangle_{L^2(M,g)}.
\end{equation}
 Now, if we put $\beta:=(P^{\mathcal{F}}+\mu)s$, \eqref{pfss} becomes $\langle (|\beta|_h,|u|_h\rangle_{L^2(M,g)}\geq \langle |(P^{\mathcal{F}}+\mu)^{-1}\beta|_h, f\rangle_{L^2(M,g)}.$ Finally, keeping in mind that $|u|_h=(\Delta_0^{\mathcal{F}}+\mu+c)^{-1}f$,  we get 
\begin{equation}
\label{leparche}
\langle (\Delta_0^{\mathcal{F}}+\mu+c)^{-1}|\beta|_h, f\rangle_{L^2(M,g)}\geq \langle |(P^{\mathcal{F}}+\mu)^{-1}\beta|_h, f\rangle_{L^2(M,g)}
\end{equation}
and therefore 
\begin{equation}
\label{leparchex}
(\Delta_0^{\mathcal{F}}+\mu+c)^{-1}|\beta|_h \geq  |(P^{\mathcal{F}}+\mu)^{-1}\beta|_h
\end{equation}
because $f$, according to Prop. \ref{disu}, is any non negative function lying in $C_c^{\infty}(M)$.
As   $s\in \mathcal{D}_0$, which by Prop. \ref{partiro} is dense in $\mathcal{D}(P^{\mathcal{F}})$ with the graph norm of $P^{\mathcal{F}}$,  we have that $\beta$ runs over a dense subset in $L^2(M,E)$   and thus \eqref{mosul} follows by the continuity of the resolvent and the map  $|\ |_h:L^2(M,E)\rightarrow L^2(M,g)$. The second statement, that is \eqref{mosulz}, follows by a general result of functional analysis, see for example \cite{HSUH} pag. 897 or Corollary 15 in the appendix of \cite{PB}. Finally \eqref{mosulzx} follows by \eqref{mosulz} applying the Trotter's product formula. See for example, \cite{HSU} pag. 31 or \cite{RSI} pag. 295-297.
\end{proof}

\section{Some general results}

This section is made of two subsections. The first collects some results concerning Sobolev spaces of sections. The second one concerns Schroedinger operators with potential bounded from below.

\subsection{Some results for Sobolev spaces and first order differential operators}

We start with the following proposition.

\begin{prop}
\label{sette}
Let $(M,g)$ be an open and  possibly incomplete Riemannian manifold. Assume that there exists a sequence of Lipschitz  functions with compact support $\{\phi_j\}_{j\in \mathbb{N}}$ such that 
\begin{itemize}
\item  $0\leq \phi_j\leq 1$ for each $j$.
\item $\phi_j\rightarrow 1 $ pointwise.
\item    $\lim_{j\rightarrow \infty}\|d_{0,\min}\phi_j\|_{L^2\Omega^1(M,g)}=0$.
\end{itemize}
Let $E$ be a vector bundle over  $M$ and let $h$ be a metric on $E$, Riemannian if $E$ is a real vector bundle, Hermitian if $E$ is a complex vector bundle. Finally let $\nabla:C^{\infty}(M,E)\rightarrow C^{\infty}(M,T^*M\otimes E)$ be a metric connection. Then we have $$W^{1,2}_0(M,E)=W^{1,2}(M,E).$$
\end{prop}

\begin{proof}
We start pointing out that by Theorem 11.3 in \cite{GYA}, the fact that $\phi_j$ has compact support for every $j\in \mathbb{N}$ and Prop. \ref{carmina}, we get that $\{\phi_j\}_{j\in \mathbb{N}}\subset \mathcal{D}(d_{0,\min})$ so that the third point in the statement is well defined.\\According to Prop. \ref{corti}, in order to prove the first point, it is enough to show that $$C^{\infty}(M,E)\cap L^{\infty}(M,E)\cap W^{1,2}(M,E)\subset W^{1,2}_0(M,E).$$ Let  $\eta_j:=d_{0,\min}\phi_j$.  
Let  $s\in C^{\infty}(M,E)\cap L^{\infty}(M,E)\cap W^{1,2}(M,E)$. Then, by Prop. \ref{pasta}, we have $\phi_js\in W^{1,2}_0(M,E)$. Moreover 
\begin{equation}
\label{zucchi}
\lim_{j\rightarrow \infty}\|s-\phi_js\|^2_{L^2(M,E)}=\lim_{j\rightarrow \infty}\int_{M}(1-\phi_j)^2|s|_h^2\dvol_g=\int_{M}\lim_{j\rightarrow \infty}(1-\phi_j)^2|s|_h^2\dvol_g=0
\end{equation}
 by the Lebesgue dominate convergence theorem.\\If now we consider $\nabla_{\min}(\phi_j s)$ then, by Prop. \ref{pasta}, we have $\nabla_{\min}(\phi_j s)=\eta_j\otimes s+\phi_j\nabla s$ and therefore $$\|\nabla s-\nabla_{\min}(\phi_js)\|_{L^2(M,T^*M\otimes E)}\leq \|\nabla s-\phi_j\nabla s\|_{L^2(M,T^*M\otimes E)}+\|\eta_j\otimes s\|_{L^2(M,T^*M\otimes E)}.$$ For the first term we have $$\lim_{j\rightarrow \infty}\|\nabla s-\phi_j\nabla s\|^2_{L^2(M,T^*M\otimes E)}=\lim_{j\rightarrow \infty}\int_{M}(1-\phi_j)^2|\nabla s|^2_{\tilde{h}}\dvol_g=0$$ again by the Lebesgue dominate convergence theorem.  For $\|\eta_j\otimes s\|^2_{L^2(M,T^*M\otimes E)}$ we have: $$\|\eta_j\otimes s\|^2_{L^2(M,T^*M\otimes E)}\leq \|\eta_j\|_{L^2\Omega^1(M,g)}^2\|s\|^2_{L^{\infty}(M,E)}$$
and therefore $$\lim_{j\rightarrow \infty}\|\eta_j\otimes s\|^2_{L^2(M,T^*M\otimes E)}=0$$ because $\lim_{j\rightarrow \infty}\|\eta_j\|^2_{L^2\Omega^1(M,g)}=0$. So we  established that 
\begin{equation}
\label{witty}
\lim_{j\rightarrow \infty}\|\nabla s-\nabla(\phi_j s)\|^2_{L^2(M, T^*M\otimes E)}=0.
\end{equation}
By  \eqref{zucchi} and \eqref{witty} we showed  that $s\in W^{1,2}_0(M,E)$ and this completes the proof.
\end{proof}

As a consequence of Prop. \eqref{sette} we have the following result.

\begin{prop}
\label{drivevz}
Let $(M,g)$ and $\{\phi_j\}_{j\in \mathbb{N}}$ be as in Prop. \ref{sette}. Let  $E$ and $F$ be two vector bundles over $M$ endowed respectively with  metrics $h$ and $\rho$, Riemannian if $E$ and $F$ are real vector bundles, Hermitian if $E$ and $F$ are complex vector bundles. Finally let  $\nabla:C^{\infty}(M,E)\rightarrow C^{\infty}(M,T^*M\otimes E)$
be a metric connection. Consider a  first order differential operator of this type:
\begin{equation}
\label{gazzz}
D:=\theta_0\circ \nabla:C^{\infty}_c(M,E)\rightarrow C^{\infty}_c(M,F)
\end{equation}
where $\theta_0\in  C^{\infty}(M,\Hom(T^*M\otimes E,F)).$ Assume that  $\theta_0$ extends as a bounded operator  $\theta:L^2(M, T^*M\otimes E)\rightarrow L^2(M, F)$.  Then we have the following inclusion:
\begin{equation}
\label{cantiano}
\mathcal{D}(D_{\max})\cap L^{\infty}(M,E)\subset \mathcal{D}(D_{\min}).
\end{equation}
In particular \eqref{cantiano} holds when $D$ is the de Rham differential $d_k:\Omega_{c}^k(M)\rightarrow \Omega^{k+1}_c(M)$, a Dirac  operator   $D:C^{\infty}_c(M,E)\rightarrow C^{\infty}_c(M,E)$ or the Dolbeault operator $\overline{\partial}_{p,q}:\Omega_{c}^{p,q}(M)\rightarrow \Omega^{p,q+1}_c(M)$ in the case $M$ is a complex manifold.
\end{prop}

\begin{proof}
The first statement, that is \eqref{cantiano}, follows arguing as in the proof of Prop. \ref{sette} and then using the continuity of $\theta:L^2(M, T^*M\otimes E)\rightarrow L^2(M, F)$. The fact that \eqref{cantiano} holds for the de Rham differential $d_k:\Omega_{c}^k(M)\rightarrow \Omega^{k+1}_c(M)$, for a Dirac  operator   $D:C^{\infty}_c(M,E)\rightarrow C^{\infty}_c(M,E)$ or for the Dolbeault operator $\overline{\partial}_{p,q}:\Omega_{c}^{p,q}(M)\rightarrow \Omega^{p,q+1}_c(M)$  is a straightforward verification.
\end{proof}

We have now the following proposition.

\begin{prop}
\label{otto}
Let $(M,g)$ be an open and possibly incomplete Riemannian manifold. Assume that for some $v\in \mathbb{R}$ with $v>2$   we have a continuous inclusion $W^{1,2}_0(M,g)\hookrightarrow L^{\frac{2v}{v-2}}(M,g)$. Let $E$ be a vector bundle over  $M$ and let $h$ be a metric on $E$, Riemannian if $E$ is a real vector bundle, Hermitian if $E$ is a complex vector bundle. Finally let $\nabla:C^{\infty}(M,E)\rightarrow C^{\infty}(M,T^*M\otimes E)$ be a metric connection. Then we have the following properties:
\begin{itemize}
\item We have a continuous inclusion $W^{1,2}_0(M,E)\hookrightarrow L^{\frac{2v}{v-2}}(M,E)$.
\item If furthermore $M$ has finite volume then the inclusion $W^{1,2}_0(M,E)\hookrightarrow L^2(M,E)$ is a compact operator. 
\end{itemize}
\end{prop}

\begin{proof}
Using Cor. \ref{katocor}, we get the continuous inclusion $$W_0^{1,2}(M,E)\hookrightarrow L^{\frac{2v}{v-2}}(M,E).$$  Now we prove the second statement. Let $\{s_i\}_{i\in \mathbb{N}}$ be a bounded sequence in $W_0^{1,2}(M,E)$. By Prop. \ref{partiro} we can assume that each $s_i$ is smooth.  Let $B\in \mathbb{R}$ be a positive number such that 
\begin{equation}
\label{anto}
\|s_i\|_{L^2(M, E)}+\|\nabla s_i\|_{L^2(M, T^*M\otimes E)}\leq B
\end{equation} 
Let $\{K_i\}_{i\in \mathbb{N}}$ be an exhausting  sequence of compact subset of $M$, that is $K_{i}\subset \inte(K_{i+1})$,  $\inte(K_i)$ is the interior of $K_i$ and $\bigcup_{i\in \mathbb{N}}K_i=M$  . Let $\{\chi_i\}_{i\in \mathbb{N}}\subset C^{\infty}_c(M)$ be a sequence of smooth functions with compact support such that
\begin{itemize}
\item $0\leq \chi_i\leq  1$,
\item $\chi_i|_{K_i}=1$,
\item $\chi_i|_{M\setminus K_{i+1}}=0$
\end{itemize} 
Then, according to Prop. \ref{pasta}, $\{\chi_1s_i\}_{i\in\mathbb{N}}$ is a bounded sequence in $W^{1,2}_0(U,E)$ where $U$ is any  open subset of $M$ with compact closure such that $\supp(\chi_1)\subset U$. Therefore, applying Prop. \ref{lunedi}, we get the existence of a subsequence of  $\{s_i\}_{i\in \mathbb{N}}$, that we label $\{s_{i,1}\}_{i\in \mathbb{N}}$, such that $\{\chi_1s_{i,1}\}_{i\in \mathbb{N}}$ converges in $L^{2}(M, E)$ to some element $\tilde{s}_1$. Now consider the sequence $\{\chi_2s_{i,1}\}_{i\in \mathbb{N}}$. Arguing as in the previous case we get the existence of a subsequence of $\{s_{i,1}\}_{i\in \mathbb{N}}$, that we label by $\{s_{i,2}\}_{i\in \mathbb{N}}$, such that $\{\chi_{2}s_{i,2}\}_{i\in \mathbb{N}}$ converges in $L^2(M, E)$ to some $\tilde{s}_2\in L^2(M, E)$. Iterating this construction we get a countable family of sequences  
\begin{equation}
\label{spesa}
\{\{\chi_1s_{i,1}\}_{i\in \mathbb{N}},\ \{\chi_2s_{i,2}\}_{i\in \mathbb{N}},...,\{\chi_ns_{i,n}\}_{i\in \mathbb{N}},... \}
\end{equation}
 such that, for each $n\in \mathbb{N}$, $\{s_{i,n+1}\}_{i\in \mathbb{N}}$ is a subsequence of $\{s_{i,n}\}_{i\in \mathbb{N}}$ and such that $\{\chi_ns_{i,n}\}_{i\in \mathbb{N}}$ converges in $L^2(M,E)$ to some element $\tilde{s}_n\in L^2(M,E)$. Now we want to prove that the sequence $\{\tilde{s}_n\}_{n\in \mathbb{N}}$ is a Cauchy sequence in $L^2(M,E)$. In order to prove this claim let $k$ and $m$ be two natural numbers with $m>k>0$. Then, by the construction performed above, we know that $\{\chi_ms_{i,m}\}_{i\in \mathbb{N}}$ converges to $\tilde{s}_m$ in $L^2(M,E)$
and that $\{\chi_ks_{i,k}\}_{i\in \mathbb{N}}$ converges to $\tilde{s}_k$ in $L^2(M,E)$. By the fact that $\{s_{i,m}\}_{i\in \mathbb{N}}$ is a subsequence of $\{s_{i,k}\}_{i\in \mathbb{N}}$ we also know that  $\{\chi_ks_{i,m}\}_{i\in \mathbb{N}}$ converges to $\tilde{s}_k$ in $L^2(M,E)$. Therefore we can find a number $j\in \mathbb{N}$ such that for each $i\geq j$ we have 
\begin{equation}
\label{topoarriva}
\max\{\|\tilde{s}_m-\chi_ms_{i,m}\|_{L^2(M,E)},\ \|\tilde{s}_k-\chi_ks_{i,m}\|_{L^2(M,E)}\}\leq \frac{1}{m}
\end{equation}
This implies that 
\begin{align}
\label{topobabi}
& \|\tilde{s}_m-\tilde{s}_k\|_{L^2(M,E)}\leq\\ 
\nn &\leq \|\tilde{s}_m-\chi_ms_{i,m}\|_{L^2(M,E)}+\|\chi_ms_{i,m}-\chi_ks_{i,m}\|_{L^2(M,E)}+\|\chi_ks_{i,m}-\tilde{s}_k\|_{L^2(M,E)}\leq\\
\nn &\leq \frac{2}{m}+ \|\chi_ms_{i,m}-\chi_ks_{i,m}\|_{L^2(M,E)}.
\end{align}
In this way, in order to conclude that   $\{\tilde{s}_n\}_{n\in \mathbb{N}}$ is a Cauchy sequence in $L^2(M,E)$, we have to estimate $\|\chi_ms_{i,m}-\chi_ks_{i,m}\|_{L^2(M,E)}$. Let $T_k$ be defined as $T_k:=M\setminus K_k$.
The fact that  $\vol_g(M)<\infty$ implies  immediately that  $\lim_{k\rightarrow \infty}\vol_g(T_k)=0$.
Then, for  $\|\chi_ms_{i,m}-\chi_ks_{i,m}\|^2_{L^2(M,E)}$, we have
$$\|\chi_ms_{i,m}-\chi_ks_{i,m}\|^2_{L^2(M,E)}=\int_{M}(\chi_m-\chi_k)^2|s_{i,m}|_h^2\dvol_g.$$
Now, in virtue of the first point of the theorem, we know that $|s_{i,m}|^2_h\in L^{\frac{v}{v-2}}(M,g)$. 
Clearly $(\chi_m-\chi_k)^2\in L^z(M,g)$ for each $z\in [1,\infty]$. Moreover $(\chi_m-\chi_k)^2=0$ on $M\setminus T_k$ and $(\chi_m-\chi_k)^2\leq 1$ on $T_k$. Therefore we can apply the H\"older inequality, see \cite{TAU} pag. 88,  and  we get  
\begin{align}
\label{babitopo}
& \int_{M}(\chi_m-\chi_k)^2|s_{i,m}|_h^2\dvol_g\leq \|(\chi_m-\chi_k)^2\|_{L^{\frac{v}{2}}(M, g))}\||s_{i,m}|^2_h\|_{L^\frac{v}{v-2}(M,g)}\leq\\
\nn & \leq \|1\|_{L^{\frac{v}{2}}(T_k, g|_{T_k}))}\||s_{i,m}|^2_h\|_{L^{\frac{v}{v-2}}(M,g)}\leq BC (\vol_g(T_k))^{\frac{2}{v}}\ \text{and}\ \lim_{k\rightarrow \infty}(\vol_g(T_k))^{\frac{2}{v}}=0
\end{align}
 where $B$ is the same constant of \eqref{anto} and $C$ is the same constant of the continuous inclusion $W_0^{1,2}(M,E)\hookrightarrow L^{\frac{2v}{v-1}}(M,E)$.  
In this way, for $\|\chi_ms_{i,m}-\chi_ks_{i,m}\|_{L^2(M,E)}$,  we get the following inequality  
\begin{equation}
\label{guetta}
\|\chi_ms_{i,m}-\chi_ks_{i,m}\|_{L^2(M,E)}\leq (BC)^{\frac{1}{2}}(\vol_g(T_k))^{\frac{1}{v}}
\end{equation}
which in turn, by \eqref{topobabi}, implies  that 
$$\|\tilde{s}_m-\tilde{s}_k\|_{L^2(M,E)}\leq \frac{2}{m}+(BC)^{\frac{1}{2}} (\vol_g(T_k))^{\frac{1}{v}}.$$ 
Therefore for each $\delta>0$ we can find $\overline{m}\in \mathbb{N}$ such that for each $m>k>\overline{m}$ we have 
\begin{equation}
\label{swed}
\|\tilde{s}_m-\tilde{s}_k\|_{L^2(M,E)}\leq \frac{2}{m}+(BC)^{\frac{1}{2}} (\vol_g(T_k))^{\frac{1}{v}}\leq \delta.
\end{equation}
In this way we proved that $\{\tilde{s}_{n}\}_{n\in \mathbb{N}}$ is a Cauchy sequence in $L^2(M,E)$ and this implies that there exists an accumulation point $s\in L^2(M,E)$ for  $\{\tilde{s}_{n}\}_{n\in \mathbb{N}}$. Finally, in order to conclude the proof, we have to show that $s$ is an accumulation point in $L^2(M,E)$ for the original sequence $\{s_i\}_{i\in \mathbb{N}}$. Let $\gamma>0$. Then we can find an element  $\tilde{s}_{m}\in \{\tilde{s}_{n}\}_{n\in \mathbb{N}}$ such that  $\|s-\tilde{s}_m\|_{L^2(M,E)}\leq \gamma$. Consider now    the countable  family of sequences  defined in \eqref{spesa}
\begin{equation}
\label{mja}
\{\{\chi_1s_{i,1}\}_{i\in \mathbb{N}}, \{\chi_2s_{i,2}\}_{i\in \mathbb{N}},...,\{\chi_ns_{i,n}\}_{i\in \mathbb{N}},...\}.
\end{equation}
We recall that,  for each $n\in \mathbb{N}$, $\{\chi_ns_{i,n}\}$ converges in $L^2(M,E)$ to $\tilde{s}_n$, that  $\{s_{i,n+1}\}_{i\in \mathbb{N}}$ is a subsequence of $\{s_{i,n}\}_{i\in \mathbb{N}}$ and that $\{s_{i,1}\}_{i\in \mathbb{N}}$ is a subsequence of $\{s_{i}\}_{i\in \mathbb{N}}$. 
Then we can find a positive integer number $i_0\in \mathbb{N}$ such that $\|\tilde{s}_m-\chi_ms_{i,m}\|_{L^2(M,E)}\leq \gamma$ for each  $i>i_0$. Now if we consider  $\|\chi_ms_{i,m}-s_{i,m}\|_{L^2(M,E)}$ then, arguing as in \eqref{babitopo}--\eqref{guetta}, we get $\|\chi_ms_{i,m}-s_{i,m}\|_{L^2(M,E)}\leq (BC)^{\frac{1}{2}}(\vol_g(T_m))^{\frac{1}{v}}$ and $(BC)^{\frac{1}{2}}(\vol_g(T_m))^{\frac{1}{v}}<\gamma$ when $m$ is sufficiently big. Ultimately for $m$ and $i$ sufficiently big we have 
\begin{align}
\nn & \|s-s_{i,m}\|_{L^2(M,E)}\leq\\
\nn &\leq  \|s-\tilde{s}_m\|_{L^2(M,E)}+\|\tilde{s}_m-\chi_ms_{i,m}\|_{L^2(M,E)}+\|\chi_ms_{i,m}-s_{i,m}\|_{L^2(M,E)}\leq\\ 
\nn &\leq 2\gamma+ (BC)^{\frac{1}{2}}(\vol_g(T_m))^{\frac{1}{v}}\leq 3\gamma.
\end{align}
This shows that $s$ is an accumulation point in $L^2(M,E)$ for the original sequence $\{s_i\}_{i\in \mathbb{N}}$ because $s_{i,m}\in \{s_i\}_{i\in \mathbb{N}}$ and therefore the proof is completed.
\end{proof}

\begin{rem}
We can reformulate the statement of  Prop. \ref{sette} saying that $\mathcal{D}(\nabla_{\max})=\mathcal{D}(\nabla_{\min})$. Analogously we can reformulate the statement of  Prop. \ref{otto} saying that there exists a continuous inclusion $\mathcal{D}(\nabla_{\min})\hookrightarrow L^{\frac{2v}{v-2}}(M,E)$ and that the natural inclusion  $\mathcal{D}(\nabla_{\min})\hookrightarrow L^{2}(M,E)$ is a compact operator where $\mathcal{D}(\nabla_{\min})$ is endowed with the corresponding graph norm.
\end{rem}

We conclude this section with the following corollary:

\begin{cor}
\label{gaber}
Consider an open and possibly incomplete Riemannian manifold $(M,g)$. Let  $E$ be a vector bundle over $M$ endowed with a metric $h$, Riemannian if $E$ is a real vector bundle, Hermitian if $E$ is a complex vector bundle. Finally let  $\nabla:C^{\infty}(M,E)\rightarrow C^{\infty}(M,T^*M\otimes E)$
be a metric connection. Assume  that the natural inclusion $W^{1,2}_0(M,E)\hookrightarrow L^2(M,E)$
is a compact operator. Then the image of $\nabla_{\min}$, $\im(\nabla_{\min})$, is a closed subset of $L^2(M,T^*M\otimes E)$.
\end{cor}

\begin{proof}
Consider the operator $\nabla^t_{\max}\circ \nabla_{\min}:L^2(M,E)\rightarrow L^2(M,E)$ whose domain is given by $\mathcal{D}(\nabla_{\max}^t\circ \nabla_{\min}):=\{s\in \mathcal{D}(\nabla_{\min}): \nabla_{\min}s\in \mathcal{D}(\nabla_{\max}^t)\}$. Arguing as in the proof of Prop. \ref{partiro} we get the following inequality $$\|\nabla_{\min}s\|^2_{L^2(M,T^*M\otimes E)}\leq \frac{1}{2}(\|s\|^2_{L^2(M,E)}+\|\nabla^t_{\max}(\nabla_{\min}s)\|^2_{L^2(M,E)})$$ for each $s\in \mathcal{D}(\nabla_{\max}^t\circ \nabla_{\min})$. Therefore we can conclude that the natural inclusion 
\begin{equation}
\label{stand}
\mathcal{D}(\nabla^t_{\max}\circ \nabla_{\min})\hookrightarrow \mathcal{D}(\nabla_{\min})
\end{equation}
is a continuous operator where each domain is endowed with the corresponding graph norm. In this way, using the assumption on $W^{1,2}_0(M,g)\hookrightarrow L^2(M,g)$, we get that the natural inclusion
\begin{equation}
\label{standa}
\mathcal{D}(\nabla^t_{\max}\circ \nabla_{\min})\hookrightarrow L^2(M,E)
\end{equation}
is a compact operator where $\mathcal{D}(\nabla^t_{\max}\circ \nabla_{\min})$ is endowed with its graph norm. As it is well know this is equivalent to say that $\nabla_{\max}^t\circ \nabla_{\min}$ is a discrete operator and this in turn implies  in particular  that it is a Fredholm operator on its domain endowed with the graph norm. Therefore we can conclude that  $\im(\nabla^t_{\max}\circ \nabla_{\min})$ is closed in $L^2(M,E)$. Now consider the two following orthogonal decompositions of $L^2(M,E)$: $$L^2(M,E)=\ker(\nabla_{\min})\oplus \overline{\im(\nabla^t_{\max})}\ \text{and}\   L^2(M,E)=\ker(\nabla^t_{\max}\circ \nabla_{\min})\oplus \overline{\im(\nabla^t_{\max}\circ \nabla_{\min})}.$$ Clearly $\ker(\nabla^t_{\max}\circ \nabla_{\min})=\ker(\nabla_{\min})$. Therefore we have the following chain of inclusions: $$\im(\nabla^t_{\max}\circ \nabla_{\min})\subset \im(\nabla^t_{\max})\subset \overline{\im(\nabla^t_{\max})}= \overline{\im(\nabla^t_{\max}\circ \nabla_{\min})}=\im(\nabla^t_{\max}\circ \nabla_{\min})$$ which in particular implies that $\overline{\im(\nabla^t_{\max})}=\im(\nabla^t_{\max})$ and therefore, taking the adjoint, $\overline{\im(\nabla_{\min})}=\im(\nabla_{\min})$. 
\end{proof}

\subsection{Some results for Schr\"odinger type operators}

\begin{prop}
\label{luglioaberlino}
Let $(M,g)$ be an open and possibly incomplete Riemannian manifold of finite volume and dimension $m>2$. Assume that we have a continuous inclusion $W^{1,2}_0(M,g)\rightarrow L^{\frac{2m}{m-2}}(M,g)$. Let $\Delta_0^{\mathcal{F}}:L^2(M,g)\hookrightarrow L^2(M,g)$ be the Friedrich extension of the Laplacian $\Delta_0:C^{\infty}_c(M)\rightarrow C^{\infty}_c(M)$. Then the heat operator $e^{-t\Delta_0^{\mathcal{F}}}:L^2(M,g)\rightarrow L^2(M,g)$ is  a trace class operator and we have the following inequalities for  its trace:
\begin{equation}
\label{intraces}
 \Tr(e^{-t\Delta_0^{\mathcal{F}}})\leq C\vol_g(M)t^{\frac{-m}{2}}
\end{equation}
for $t\in (0,1)$.
\end{prop}

\begin{proof}
As showed in \cite{ACM} pag. 1062  or in  \cite{KIY} Theorem 2.1  the continuous inclusion $W^{1,2}_0(M,g)\hookrightarrow L^{\frac{2m}{m-2}}(M,g)$ is equivalent to the following property: 
\begin{equation}
\label{tracce}
k_{\Delta_0}(t,x,y)\leq Ct^{\frac{-m}{2}}
\end{equation}
 for each $(x,y)\in M\times M$, $t\in (0,1)$ and where $k_{\Delta_0}(t,x,y)$ is the smooth kernel of the heat operator $e^{-t\Delta_0^{\mathcal{F}}}$ and $C$ is a positive constant. This implies immediately that $\tr(e^{-t\Delta_0^{\mathcal{F}}})\leq C t^{\frac{-m}{2}}$ for $t\in (0,1)$. Moreover using \eqref{tracce} we get that $$\int_{M\times M}(k_{\Delta_0}(t,x,y))^2\dvol_g(x)\dvol_g(v)<\infty$$ for each $t\in (0,1)$ and this in turn implies that $e^{-t\Delta^{\mathcal{F}}}:L^2(M,g)\rightarrow L^2(M,g)$ is a Hilbert-Schimdt operator for each $t\in (0,1)$, see for instance \cite{RSI} pag. 210. Writing $t=\frac{t}{2}+\frac{t}{2}$ we get $e^{-t\Delta^{\mathcal{F}}}=e^{-\frac{t}{2}\Delta^{\mathcal{F}}}\circ e^{-\frac{t}{2}\Delta^{\mathcal{F}}}$ and this tells us that  $e^{-t\Delta^{\mathcal{F}}}:L^2(M,g)\rightarrow L^2(M,g)$ is a trace class operator for each $t\in (0,1)$. Ultimately, fixing any $t>0$ and writing $t=n\frac{t}{n}$ with $n>t$, we have $e^{-t\Delta^{\mathcal{F}}}=e^{-\frac{t}{n}\Delta^{\mathcal{F}}}\circ...\circ e^{-\frac{t}{n}\Delta^{\mathcal{F}}}$ ($n$-times) and this allows us to conclude that $e^{-t\Delta^{\mathcal{F}}}:L^2(M,g)\rightarrow L^2(M,g)$ is a trace class operator. Finally we have $\Tr(e^{-t\Delta^{\mathcal{F}}})=\int_Mk_{\Delta_0}(t,x,x)\dvol_g$ and therefore for each $t\in (0,1)$ we find $$\Tr(e^{-t\Delta^{\mathcal{F}}})=\int_Mk_{\Delta_0}(t,x,x)\dvol_g\leq \int_MCt^{-\frac{m}{2}}\dvol_g=C\vol_g(M)t^{\frac{-m}{2}}.$$
\end{proof}

As in the previous proposition consider again a possibly incomplete Riemannian manifold $(M,g)$ of finite volume and dimension $m$. Let  $E$ be a vector bundle over $M$ endowed with a metric $h$, Riemannian if $E$ is a real vector bundle, Hermitian if $E$ is a complex vector bundle.  Finally let $\nabla:C^{\infty}(M,E)\rightarrow C^{\infty}(M,T^*M\otimes E)$ be a metric connection. Our goal is to study  some properties of Schr\"odinger type operators, that is  operators of type 
\begin{equation}
\label{polverec}
\nabla^t\circ\nabla +L 
\end{equation}
 where $\nabla^t: C_c^{\infty}(M,T^*M\otimes E)\rightarrow C_c^{\infty}(M,E)$ is the formal adjoint of $\nabla$ and $L\in C^{\infty}(M,\End(E))$.  In particular we are interested in  \eqref{polverec} acting on $L^2(M,E)$ with $C^{\infty}_c(M,E)$ as core domain. One of the reasons, as it is well known,  is provided by the fact that the square of the typical first order  differential operators arising in differential geometry, for instance  the Gauss-Bonnet operator $d+\delta$, the Hodge-Dolbeault operator $\overline{\partial}+\overline{\partial}^t$ and  the Spin-Dirac operator $\eth$, are Schr\"odinger type operators.  

\begin{prop}
\label{skernelz}
Let $(M,g)$ be an open and possibly incomplete Riemannian manifold of finite volume and dimension $m>2$. Assume that we have a continuous inclusion $W^{1,2}_0(M,g)\rightarrow L^{\frac{2m}{m-2}}(M,g)$. Let $V$, $E$, $g$, $h$, and $\nabla$ be as described above. Let $$P:=\nabla^t\circ \nabla +L,\ P:C^{\infty}_c(M,E)\rightarrow C_c^{\infty}(M,E)$$ be a Schr\"odinger type operator with $L\in C^{\infty}(M,\End(E))$. Assume that:
\begin{itemize}
\item $P$ is symmetric and  positive.
\item    There is a constant $c\in \mathbb{R}$ such that, for each $s\in C^{\infty}(M,E)$, we have  $$h(Ls,s)\geq ch(s,s).$$
\end{itemize}
Let $P^{\mathcal{F}}:L^2(M,E)\rightarrow L^2(M,E)$ be the Friedrich extension of $P$ and let  $\Delta_0^{\mathcal{F}}:L^2(M,g)\rightarrow L^2(M,g)$  be the Friedrich extension of $\Delta_0:C^{\infty}_c(M)\rightarrow C^{\infty}_c(M)$. Then the heat operator associated to $P^{\mathcal{F}}$ $$e^{-tP^{\mathcal{F}}}:L^{2}(M,E)\longrightarrow L^2(M,E)$$  
is a trace class operator and its trace satisfies the following inequality: 
\begin{equation}
\label{marzz}
\Tr(e^{-tP^{\mathcal{F}}})\leq re^{-tc}\Tr(e^{-t\Delta_0^{\mathcal{F}}}).
\end{equation}
where $r$ is the rank of the vector bundle $E$.
\end{prop}

\begin{proof}
According to Theorem \ref{ready} we know that 
\begin{equation}
\label{sargon}
|e^{-tP^{\mathcal{F}}}s|_h\leq e^{-tc}e^{-\Delta_0^{\mathcal{F}}}|s|_h
\end{equation}
for each $s\in L^2(M,E).$ Let $k_P(t,x,y)$ be the smooth kernel of the heat operator $e^{-tP^{\mathcal{F}}}$ and analogously let  $k_{\Delta_0}(t,x,y)$ be the smooth kernel of the heat operator $e^{-t\Delta_0^{\mathcal{F}}}$. Therefore, for each pair $(x,y)\in M\times M$, $k_P(t,x,y)\in \Hom(E_y,E_x)$ and analogously $k_{\Delta_0}(t,x,y)\in \Hom(\mathbb{R},\mathbb{R})$ that is $k_{\Delta_0}(t,x,y)\in C^{\infty}(M\times M)$.   Let us label by $\|\ \|_{h,\op}$ the pointwise operator norm for the linear operators acting between  $(E_y,h_y)$ and $(E_x,h_x)$. As explained in \cite{HSU} pag. 32, the inequality \eqref{sargon} implies that  for the pointwise operator norm $\|k_P(t,x,y)\|_{h,\op}$ the following inequality holds: $$\|k_P(t,x,y)\|_{h,\op}\leq e^{-tc}k_{\Delta_0}(t,x,y).$$ In particular for $x=y$ we have  
\begin{equation}
\label{shugar}
\|k_P(t,x,x)\|_{h,\op}\leq e^{-tc}k_{\Delta_0}(t,x,x).
\end{equation}
 This implies immediately the following inequality for the pointwise  traces 
\begin{equation}
\label{frazier}
\tr(k_P(t,x,x))\leq r e^{-tc}k_{\Delta_0}(t,x,x).
\end{equation} 
By Prop. \ref{luglioaberlino} we know that $e^{-t\Delta_0^{\mathcal{F}}}:L^2(M,g)\rightarrow L^2(M,g)$ is a trace class operator. Therefore, using  \eqref{shugar} and \eqref{frazier}, we get that also $e^{-tP^{\mathcal{F}}}:L^2(M,E)\rightarrow L^2(M,E)$ is a trace class operator and  its trace satisfies the inequality
\begin{equation}
\label{ronlyle}
\Tr(e^{-tP^{\mathcal{F}}})\leq r e^{-tc}\Tr(e^{-t\Delta_0^{\mathcal{F}}}).
\end{equation}
\end{proof}

\begin{cor}
\label{eigen}
Under the assumptions of Prop. \ref{skernelz}. For each $t\in (0,1)$ we have the following inequalities for the pointwise trace and for the heat trace of $e^{-tP^{\mathcal{F}}}$ respectively:
$$\tr(k_{P}(t,x,x))\leq Cre^{-tc}t^{\frac{-m}{2}}$$ $$\Tr(e^{-tP^{\mathcal{F}}})\leq Cre^{-tc}\vol_g(M)t^{\frac{-m}{2}}$$ where $C$ is the same constant of \eqref{intraces}. The operator $P^{\mathcal{F}}:L^2(M,E)\rightarrow L^2(M,E)$ is a discrete operator. 
If we label its eigenvalues with $$0\leq \lambda_{0}\leq \lambda_1\leq...\leq \lambda_n\leq...$$ then there exists a positive constant $K$ such that  we have the following asymptotic inequality 
\begin{equation}
\label{sierra}
\lambda_j\geq Kj^{\frac{2}{m}}+c
\end{equation}
 as $j\rightarrow \infty$.
\end{cor}

\begin{proof}
The  inequality for the pointwise trace as well as  that for the heat trace of $e^{-tP^{\mathcal{F}}}$ follow by Prop. \ref{luglioaberlino} and Prop. \ref{skernelz}. By the fact that $e^{-tP^\mathcal{F}}$ is a trace class operator we get immediately that $P^{\mathcal{F}}$ is a discrete operator.  Finally, by \eqref{marzz} and the first point of this corollary, we know that $\sum_je^{-t\lambda_j}\leq Cre^{-tc}\vol_g(M)t^{\frac{-m}{2}}.$ This is equivalent to say that  $$\sum_{j\in \mathbb{N}}e^{-t\mu_j}\leq Cr\vol_g(M)t^{\frac{-m}{2}}$$ where $\mu_j:=\lambda_j-c$. Now the thesis follows applying a classical argument from Tauberian theory, see for instance \cite{Tay} pag. 107.
\end{proof}

\begin{prop}
\label{pietrobz}
Under the assumptions of Theorem \ref{skernelz}.  Let $k_{P}(t,x,y)$ and $\|k_{P}(t,x,y)\|_{h,\op}$ be as in the proof of Theorem \ref{skernelz}. Then the following inequality holds for $0<t<1$:
\begin{equation}
\label{ubez}
\|k_{P}(t,x,y)\|_{h,\op}\leq Ce^{-tc}t^{\frac{-m}{2}}.
\end{equation}
where $C$ is the same  positive constant of \eqref{intraces}. This implies that
\begin{enumerate}
\item $e^{-tP^{\mathcal{F}}}$ is a ultracontractive operator for each $0<t<1$. This means, see \cite{LSC}, that for each $0<t<1$ there exists $C_t>0$  such that   $$\|e^{-tP^{\mathcal{F}}}s\|_{L^{\infty}(M,E)}\leq C_t\|s\|_{L^1(M,E)}$$ for each $s\in L^{1}(M,E)$. In particular, for each $0<t<1$, $e^{-tP^{\mathcal{F}}}:L^{1}(M,E)\rightarrow L^{\infty}(M,E)$ is continuous.
\item If $s$ is an eigensection of $P^\mathcal{F}:L^2(M,E)\rightarrow L^2(M,E)$ then $s\in L^{\infty}(M,E)$.
\end{enumerate}
\end{prop}

\begin{proof}
As pointed out in the proof of Prop. \ref{skernelz} we have $\|k_P(t,x,y)\|_{h,\op}\leq e^{-ct}k_{\Delta_0}(t,x,y)$. 
By the assumptions we know that there is a continuous inclusion $W_0^{1,2}(M,g)\hookrightarrow L^{\frac{2m}{m-2}}(M,g)$. As recalled in the proof of Prop. \ref{luglioaberlino} this is equivalent to say that, for some positive constant $C$, $$k_{\Delta_0}(t,x,y)\leq Ct^{\frac{-m}{2}},\ 0<t<1.$$ Combining together the previous inequalities we have, for $0<t<1$, $$\|k_P(t,x,y)\|_{h,\op}\leq Ce^{-tc}t^{\frac{-m}{2}}$$ and this establishes \eqref{ubez}.
Finally the remaining two properties follow immediately using \eqref{ubez}.
\end{proof}

\section{Applications to irreducible complex projective varieties}

\subsection{Sobolev spaces on irreducible complex projective varieties}

This section  concerns {\em irreducible complex projective varieties $V\subset \mathbb{C}\mathbb{P}^n$}. This means that $V$ is the zero set of a family of homogeneous polynomials such that it is not possible to decompose $V$ as $V=V_1\cup V_2$ with $V_1\subset V$, $V_2\subset V$, $V\neq V_1$, $V\neq V_2$ and such that $V_1$ and $V_2$ are the zero set of other two families of homogeneous polynomials. Using the language of Zariski topology this means that $V$ is a Zariski closed subset of $\mathbb{C}\mathbb{P}^n$ and it is not possible to decompose $V$  as $V=V_1\cup V_2$ with $V_1\subset V$, $V_2\subset V$, $V\neq V_1$, $V\neq V_2$ where $V_1$ and $V_2$ are other two Zariski closed subsets of $\mathbb{C}\mathbb{P}^n$. Our reference for this topic is \cite{GHa}. Given an irreducible complex projective variety $V\subset \mathbb{C}\mathbb{P}^n$ we will label by  $\sing(V)$ the singular locus of $V$ and by $\reg(V):= V\setminus \sing(V)$ the regular part of $V$. The regular part of $V$, $\reg(V)$,  becomes a K\"ahler manifold when we endow it with the K\"ahler metric induced by the Fubini-Study metric of $\mathbb{C}\mathbb{P}^n$. In particular we get an {\em open and incomplete K\"ahler manifold} when $\sing(V)\neq \emptyset$.\\Now we state a proposition which provides the existence of a suitable sequence of cut-off functions. A similar result is contained   in  \cite{LT} pag. 871 and in \cite{KIY}, Theorem 3.1 and Theorem 3.2.  First we recall the following property.

\begin{prop}
\label{milma}
Let $M$ be a complex manifold and let $h$ and $g$ be two Hermitian metrics on $M$ such that $g\geq h$. Then for each $\eta \in \Omega^1_c(M)$ we have $\|\eta\|_{L^2\Omega^1(M,h)}\leq \|\eta\|_{L^2\Omega^1(M,g)}$. Therefore the identity map $\id: \Omega^1_c(M)\rightarrow \Omega^1_c(M)$ extends  as a continuous inclusion $L^2\Omega^1(M,g)\hookrightarrow L^2\Omega^1(M,h)$ so that for each $\phi\in L^2\Omega^1(M,g)$ we have $\|\phi\|_{L^2\Omega^1(M,h)}\leq \|\phi\|_{L^2\Omega^1(M,g)}$.
\end{prop}
\begin{proof}
The proof lies on a careful calculation of linear algebra. It is carried out, for instance, in \cite{GMMI} pag. 146.
\end{proof}

\begin{prop}
\label{taglia}
Let $V\subset \mathbb{C}\mathbb{P}^n$ be an irreducible  complex projective variety of complex dimension $v$ and  let $g$ be the K\"ahler  metric on $\reg(V)$ induced by the Fubini-Study metric of $\mathbb{C}\mathbb{P}^n$.
Then there exists a sequence of Lipschitz  functions  $\{\phi_j\}_{j\in \mathbb{N}}$ with compact support in $\reg(V)$ such that 
\begin{itemize}
\item  $0\leq \phi_j\leq 1$ for each $j$.
\item $\phi_j\rightarrow 1 $ pointwise.
\item $\phi_j\in \mathcal{D}(d_{0,\min})$ for each $j\in \mathbb{N}$ and  $\lim_{j\rightarrow \infty}\|d_{0,\min}\phi_j\|_{L^2\Omega^1(\reg(V),g)}=0$.  
\end{itemize}
In particular $1\in \mathcal{D}(d_{0,\min})$.
\end{prop}

\begin{proof}
Let $\pi:\tilde{V}\longrightarrow V$ be a resolution of singularities (which exists thanks to the fundamental work of Hironaka, 
see \cite{Hiro}).  We recall that $\pi:\tilde{V}\longrightarrow V$ is a holomorphic and surjective map such that $$\pi|_{\tilde{V}\setminus E}:\tilde{V}\setminus E\longrightarrow V\setminus \sing(V)$$ is a biholomorphism  where $E=\pi^{-1}(\sing(V))$ is the exceptional set. Moreover we can  assume that $E$ is a divisor with only normal crossings, that is, the irreducible components of $E$ are regular and meet complex transversely. In particular, and this is what we need for our purpose,  $\tilde{V}\setminus \pi^{-1}(\reg(V))$ is a finite union of compact complex submanifolds, that is $\tilde{V}\setminus \pi^{-1}(\reg(V))=\cup_{i=1}^m S_i$ for some $m\in \mathbb{N}$. Therefore for  the real codimension of $S_i$ we have $\cod_{\mathbb{R}}(S_i)\geq 2$ for each $i=1,...,m$. Consider now a Hermitian metric $h$ on $\tilde{V}$. Let us define $V':=\pi^{-1}(\reg(V))$ and let $h'$ be defined us $h':=h|_{V'}$. As a first step we want to show that on $(V',h')$ there is a sequence of Lipschitz functions with compact support $\{\psi_j\}_{j\in \mathbb{N}}$ which satisfies the three properties stated in  this proposition. To this aim we adapt to our case the strategy used in \cite{ChFe} and in \cite{LT}. Define $M_i:= \tilde{V}\setminus  S_i$.
Let $r_i$ be the distance function to $S_i$ induced by $h$. Let $\epsilon_n:=\frac{1}{n^2}$ and let $\epsilon'_n:=e^{-\frac{1}{\epsilon_n^2}}=e^{-n^4}$. Then we define $\psi_{j,M_i}$ as 
\begin{equation}
\label{cut}
\psi_{j,M_i}:= \left\{
\begin{array}{ll}
1 & r_i\geq \epsilon_n\\
(\frac{r_i}{\epsilon_n})^{\epsilon_n} & 2\epsilon'_n\leq r_i\leq \epsilon_n\\
(\frac{2\epsilon'_n}{\epsilon_n})^{\epsilon_n}(\frac{r_i}{\epsilon'_n}-1) &  \epsilon'_n\leq r_i\leq 2\epsilon'_n\\
0 & 0\leq r_i\leq \epsilon'_n
\end{array}
\right.
\end{equation}

By \eqref{cut} we get  easily that each $\psi_{j,M_i}$ is a Lipschitz function with compact support. Therefore, combining together Theorem 11.3 in \cite{GYA}, Prop. \ref{carmina} and the fact that $M_i$ has finite volume, we get that $\{\psi_{j,M_i}\}_{j\in \mathbb{N}}\subset \mathcal{D}(d_{0,\min})$ on $(M_i,h|_{M_i})$.  Clearly we have $0\leq \psi_{j,M_i}\leq 1$ and $\lim_{j\rightarrow \infty}\psi_{j,M_i}=1$ pointwise. Moreover, according to \cite{ChFe},  we have 
\begin{equation}
\label{paspas}
\lim_{j\rightarrow \infty}\|d_{0,\min}\psi_{j,M_i}\|_{L^2\Omega^1(M_i,h|_{M_i})}=0.
\end{equation}
 Here, we only recall that the previous limit is  based on an estimate of the volume of a tubular neighborhood of $S_i$ and that for this estimate the lower bound on the real codimension of $S_i$ plays a fundamental role. Now define $$\psi_j:=\prod_{i=1}^m \psi_{j,M_{i}}.$$ For each $j\in \mathbb{N}$,  $\psi_j$ is defined as a product of a finite number of non negative Lipschitz functions with compact support and bounded above by $1$.  Therefore $\psi_j$ is in turn a non negative Lipschitz function which compact support and bounded above by $1$. Thus, arguing as above, we can conclude that $\{\psi_j\}_{j\in \mathbb{N}}\subset \mathcal{D}(d_{0,\min})$ on $(V',h')$.   Clearly  for each $\psi_j$ we have $0\leq \psi_{j}\leq 1$ and $\lim_{j\rightarrow \infty}\psi_{j}=1$ pointwise.
Now we have to show that 
\begin{equation}
\label{limit}
\lim_{j\rightarrow \infty}\langle d_{0,\min} \psi_j,d_{0,\min} \psi_j\rangle_{L^2\Omega^1(V',h')}=0.
\end{equation}
We have $d_{0,\min} \psi_j=\sum_{i=1}^m\gamma_id_{0,\min}\psi_{j,M_i}$ where $\gamma_i$ is given by the product $\psi_{j,M_1}...\psi_{j,M_{i-1}}\psi_{j,M_{i+1}}...\psi_{j,M_m}$. By the fact that $0\leq \gamma_i\leq 1$ to establish \eqref{limit} it is enough to show that 
$$\lim_{j\rightarrow \infty}\langle d_{0,\min} \psi_{j,M_p},d_{0,\min} \psi_{j,M_q}\rangle_{L^2\Omega^1(V',h')}=0\ \text{for each }\ p,q\in \{1,...,m\}.
$$
This follows because 
$$
\langle d_{0,\min} \psi_{j,M_p},d_{0,\min} \psi_{j,M_q}\rangle_{L^2\Omega^1(V',h')} \leq \|d_{0,\min} \psi_{j,M_p}\|_{L^2\Omega^1(V',h')} \|d_{0,\min} \psi_{j,M_q}\|_{L^2\Omega^1(V',h')},
$$
and, by \eqref{paspas}, we have 
$$
 \lim_{j\rightarrow \infty}\|d_{0,\min} \psi_{j,M_p}\|_{L^2\Omega^1(V',h')}=0,
$$
and 
$$
\lim_{j\rightarrow \infty}\|d_{0,\min} \psi_{j,M_q}\|_{L^2\Omega^1(V',h')}=0.
$$
This allow us to conclude  that on $(V',h')$ there is a sequence of Lipschitz functions with compact support $\{\psi_j\}_{j\in \mathbb{N}}$ which satisfies the three properties stated in  this proposition. Now let $\tilde{g}$ be the K\"ahler metric on $V'$  defined as $\pi^*g$. We can look at $\tilde{g}$ as the pull-back of the Fubini-Study metric on $\mathbb{C}\mathbb{P}^n$ through the map $\pi:\tilde{V}\longrightarrow \mathbb{C}\mathbb{P}^n$. By the fact that   $d\pi$, the differential of $\pi$, degenerates on $\tilde{V}\setminus V'$ we get that $\tilde{g}\leq Ch'$,  for some positive real constant $C>0$. Now, as an immediate application of Prop. \ref{milma}, we can conclude  that the sequence $\{\psi_j\}_{j\in \mathbb{N}}$ satisfies the three properties stated in  this proposition also with respect to the K\"ahler  manifold $(V',\tilde{g})$. Finally, by the fact that $\pi|_{V'}:(V',\tilde{g})\longrightarrow (\reg(V),g)$ is an isometry, defining $\phi_j:=\psi_j\circ (\pi|_{V'})^{-1}$ we obtain our desired sequence on $(\reg(V),g)$.
\end{proof}

Now we have the main result of this section.

\begin{teo}
\label{pianopiano}
Let $V\subset \mathbb{C}\mathbb{P}^n$ be an irreducible complex projective variety of complex dimension $v$. Let $E$ be a vector bundle over $\reg(V)$ and let $h$ be a metric on $E$, Riemannian if $E$ is a real vector bundle, Hermitian if $E$ is a complex vector bundle. Let $g$ be the K\"ahler  metric on $\reg(V)$ induced by the Fubini-Study metric of  $\mathbb{C}\mathbb{P}^n$. Finally let $\nabla:C^{\infty}(\reg(V),E)\rightarrow C^{\infty}(\reg(V),T^*\reg(V)\otimes E)$ be a metric connection. We have the following properties:
\begin{itemize}
\item $W^{1,2}(\reg(V),E)=W^{1,2}_0(\reg(V),E)$.
\item Assume that $v>1$. Then there exists a continuous inclusion $W^{1,2}(\reg(V),E)\hookrightarrow L^{\frac{2v}{v-1}}(\reg(V),E)$.
\item Assume that $v>1$. Then the inclusion $W^{1,2}(\reg(V),E)\hookrightarrow L^2(\reg(V),E)$ is a compact operator.
\end{itemize}
\end{teo}

\begin{proof}
The first point follows by Prop. \ref{taglia} and by  Prop. \ref{sette}.  The continuous inclusion 
$W^{1,2}_0(\reg(V),g)\hookrightarrow L^{\frac{2v}{v-1}}(\reg(V),g)$ is established in \cite{LT} pag. 874 or in \cite{KIY} pag. 113. Now, by the first point of this theorem (or by \cite{LT} Theorem 4.1 or by \cite{KIY} Cor. 3.1), we know that $W^{1,2}(\reg(V),g)=W_0^{1,2}(\reg(V),g)$ and therefore we have the continuous inclusion $W^{1,2}(\reg(V),g)\hookrightarrow L^{\frac{2v}{v-1}}(\reg(V),g)$. Now, using Prop. \ref{devo}, we get the continuous inclusion $$C^{\infty}(\reg(V),E)\cap W^{1,2}(\reg(V),E)\hookrightarrow L^{\frac{2v}{v-1}}(\reg(V),E).$$ Finally, by the density of $C^{\infty}(\reg(V),E)\cap W^{1,2}(\reg(V),E)$ in $W^{1,2}(\reg(V),E)$, see Prop. \ref{carmina}, the continuous inclusion $W^{1,2}(\reg(V),E)\hookrightarrow L^{\frac{2v}{v-1}}(\reg(V),E)$ is established. Finally the third point is a consequence of the second point and  Prop. \ref{otto}.
\end{proof}

\begin{rem}
We can reformulate the statement of  Theorem \ref{pianopiano} saying that $\mathcal{D}(\nabla_{\max})=\mathcal{D}(\nabla_{\min})$, there exists a continuous inclusion $\mathcal{D}(\nabla_{\max})\hookrightarrow L^{\frac{2v}{v-1}}(\reg(V),E)$ and that the natural inclusion  $\mathcal{D}(\nabla_{\max})\hookrightarrow L^{2}(\reg(V),E)$ is a compact operator where $\mathcal{D}(\nabla_{\max})$ is endowed with the corresponding graph norm.
\end{rem}

\begin{cor}
\label{delida}
Under the assumptions of Theorem \ref{pianopiano}. Then $\im(\nabla_{\min})=\im(\nabla_{\max})$ is a closed subspace of $L^2(\reg(V),T^*\reg(V)\otimes E)$.
\end{cor}

\begin{proof}
According to Prop. \ref{pianopiano} we know that $\nabla_{\max}=\nabla_{\min}$ and therefore $\im(\nabla_{\max})=\im(\nabla_{\min})$. Now the thesis follows by Cor. \ref{gaber}.
\end{proof}

We conclude this section with the following proposition. The case of the Dolbeault operator is already treated in \cite{JRU}.

\begin{prop}
\label{nove}
Let $(\reg(V),g)$ be as in Theorem \ref{pianopiano}. Let  $E$ and $F$ be two vector bundles over $\reg(V)$ endowed respectively with  metrics $h$ and $\rho$, Riemannian if $E$ and $F$ are real vector bundles, Hermitian if $E$ and $F$ are complex vector bundles. Finally let  $\nabla:C^{\infty}(\reg(V),E)\rightarrow C^{\infty}(M,T^*\reg(V)\otimes E)$
be a metric connection. Consider a  first order differential operator of this type:
\begin{equation}
\label{gazzzz}
D:=\theta_0\circ \nabla:C^{\infty}_c(\reg(V),E)\rightarrow C^{\infty}_c(\reg(V),F)
\end{equation}
where $\theta_0\in  C^{\infty}(\reg(V),\Hom(T^*\reg(V)\otimes E,F)).$ Assume that  $\theta_0$ extends as a bounded operator  $$\theta:L^2(\reg(V), T^*\reg(V)\otimes E)\rightarrow L^2(\reg(V), F).$$  Then we have the following inclusion:
\begin{equation}
\label{cantianobello}
\mathcal{D}(D_{\max})\cap L^{\infty}(\reg(V),E)\subset \mathcal{D}(D_{\min}).
\end{equation}
In particular \eqref{cantianobello} holds when $D$ is the de Rham differential $d_k:\Omega_{c}^k(\reg(V))\rightarrow \Omega^{k+1}_c(\reg(V))$, a Dirac  operator   $D:C^{\infty}_c(\reg(V),E)\rightarrow C^{\infty}_c(\reg(V),E)$ or the Dolbeault operator $\overline{\partial}_{p,q}:\Omega_{c}^{p,q}(\reg(V))\rightarrow \Omega^{p,q+1}_c(\reg(V))$.
\end{prop}

\begin{proof}
This follows applying Theorem \ref{pianopiano} and Prop. \ref{drivevz}.
\end{proof}

\subsection{Schr\"odinger operators on irreducible complex projective varieties}

As in the previous section consider again an irreducible complex projective variety $V\subset \mathbb{C}\mathbb{P}^n$ of complex dimension $v$. Let $\reg(V)$ be its regular part and let $E$ be a vector bundle over $\reg(V)$ endowed with a metric $h$, Riemannian if $E$ is a real vector bundle, ermitian if $E$ is a complex vector bundle. Finally let $g$ be the K\"ahler metric on $\reg(V)$ induced by the Fubini-Study metric of  $\mathbb{C}\mathbb{P}^n$ and let $\nabla:C^{\infty}(\reg(V),E)\rightarrow C^{\infty}(\reg(V),T^*M\otimes E)$ be a metric connection. In this section we consider again some  Schr\"odinger type operators
\begin{equation}
\label{polvere}
\nabla^t\circ\nabla +L 
\end{equation}
 where $\nabla^t: C_c^{\infty}(\reg(V),T^*M\otimes E)\rightarrow C_c^{\infty}(\reg(V),E)$ is the formal adjoint of $\nabla$ and $L\in C^{\infty}(\reg(V),\End(E))$.

\begin{teo}
\label{skernel}
Let $V$, $E$, $g$, $h$, and $\nabla$ be as described above. Let $$P:=\nabla^t\circ \nabla +L,\ P:C^{\infty}_c(\reg(V),E)\rightarrow C_c^{\infty}(\reg(V),E)$$ be a Schr\"odinger type operator with $L\in C^{\infty}(\reg(V),\End(E))$. Assume that:
\begin{itemize}
\item $P$ is symmetric and  positive.
\item    There is a constant $c\in \mathbb{R}$ such that, for each $s\in C^{\infty}(\reg(V),E)$, we have  $$h(Ls,s)\geq ch(s,s).$$
\end{itemize}
Let $P^{\mathcal{F}}:L^2(\reg(V),E)\rightarrow L^2(\reg(V),E)$ be the Friedrich extension of $P$ and let  $\Delta_0^{\mathcal{F}}:L^2(\reg(V),g)\rightarrow L^2(\reg(V),g)$  be the Friedrich extension of $\Delta_0:C^{\infty}_c(\reg(V))\rightarrow C^{\infty}_c(\reg(V))$. Then the heat operator associated to $P^{\mathcal{F}}$ $$e^{-tP^{\mathcal{F}}}:L^{2}(\reg(V),E)\longrightarrow L^2(\reg(V),E)$$  
is a trace class operator and its trace satisfies the following inequality: 
\begin{equation}
\label{marz}
\Tr(e^{-tP^{\mathcal{F}}})\leq me^{-tc}\Tr(e^{-t\Delta_0^{\mathcal{F}}}).
\end{equation}
where $m$ is the rank of the vector bundle $E$.
\end{teo}

\begin{proof}
This follows by Prop. \ref{skernelz} and by Theorem \ref{pianopiano}.
\end{proof}

\begin{cor}
\label{quovadis}
Under the assumptions of Theorem \ref{skernel}. The operator  $P^{\mathcal{F}}:L^2(\reg(V),E)\rightarrow L^2(\reg(V),E)$  is a discrete operator.  Moreover, for $t\in (0,1)$, we have the following inequalities:
\begin{equation}
\label{blau}
\tr(k_P(t,x,x))\leq me^{-tc}(4\pi t)^{-v}(1+\frac{4v(v+1)}{6}t+O(t^2))
\end{equation}
\begin{equation}
\label{zun}
\Tr(e^{-tP^{\mathcal{F}}})\leq me^{-tc}(4\pi t)^{-v}\left(\vol_g(\reg(V))(1+\frac{4v(v+1)}{6}t)+O(t^2)\right).
\end{equation}
Let $\{\lambda_j\}$ be the sequence of eigenvalues of  $P^{\mathcal{F}}:L^2(\reg(V),E)\rightarrow L^2(\reg(V),E)$. Then we have  the following asymptotic inequality:
\begin{equation}
\label{waser}
\lambda_j\geq  \left(\frac{(2\pi)^{2v}j}{\omega_{2v}m\vol_g(\reg(V))}\right)^{\frac{1}{v}}+c
\end{equation}
as $j\rightarrow \infty$ where $\omega_{2v}$ is the volume of the unit $2v$-ball in $\mathbb{R}^{2v}$.
\end{cor}
\begin{proof}
The fact that $P^{\mathcal{F}}:L^2(\reg(V),E)\rightarrow L^2(\reg(V),E)$ is a discrete operator is  a consequence of Theorem \ref{skernel}.  Inequalities \eqref{blau}, \eqref{zun} and \eqref{waser}  follow by Theorem \ref{skernel} and Corollary 5.4 in \cite{LT}.  
\end{proof}

An important case of the previous corollary is given by  $(\nabla^t\circ \nabla)^{\mathcal {F}}$, the Friedrich extension of the Bochner Laplacian $\nabla^t\circ \nabla:C^{\infty}_c(\reg(V),E)\rightarrow C^{\infty}_c(\reg(V),E)$. As before we label by $k_{\nabla^t\circ \nabla}(t,x,y)$ and by $k_{\Delta_0}(t,x,y)$ the smooth  kernel of the heat operators $e^{-t(\nabla^t\circ \nabla)^{\mathcal{F}}}$ and $e^{-t\Delta_0^{\mathcal{F}}}$ respectively.
\begin{cor}
\label{demo}
Let $V, E, h$ and $g$ as in the statement of  Theorem \ref{skernel}. Consider the Bochner Laplacian $$\nabla^t\circ \nabla:C^{\infty}_c(\reg(V),E)\rightarrow C^{\infty}_c(\reg(V),E).$$ Let 
\begin{equation}
\label{boclap}
(\nabla^t\circ \nabla)^{\mathcal{F}}:L^2(\reg(V),E)\rightarrow L^2(\reg(V),E)
\end{equation}
be its Friedrich extension. Then 
\begin{equation}
\label{muio}
e^{-t(\nabla^t\circ \nabla)^{\mathcal{F}}}:L^{2}(\reg(V),E)\longrightarrow L^2(\reg(V),E)
\end{equation}
is a trace class operator;  its pointwise trace and its  trace satisfy respectively the following inequalities:
\begin{equation}
\label{jfrazier}
\tr(k_{\nabla^t\circ \nabla}(t,x,x))\leq m(4\pi t)^{-v}(1+\frac{4v(v+1)}{6}t+O(t^2))
\end{equation} 
 \begin{equation}
\label{marzl}
\Tr(e^{-t(\nabla^t\circ \nabla)^{\mathcal{F}}})\leq m(4\pi t)^{-v}\left(\vol_g(\reg(V))(1+\frac{4v(v+1)}{6}t)+O(t^2)\right)
\end{equation}
for $t\in (0,1)$. Furthermore \eqref{boclap} is a discrete operator and its sequence of eigenvalues, $\{\lambda_j\}$,  satisfies the following asymptotic inequality:
\begin{equation}
\label{waserz}
\lambda_j\geq  \left(\frac{(2\pi)^{2v}j}{\omega_{2v}m\vol_g(\reg(V))}\right)^{\frac{1}{v}}
\end{equation}
as $j\rightarrow \infty$ where $\omega_{2v}$ is the volume of the unit $2v$-ball in $\mathbb{R}^{2v}$.
Finally a  core domain for \eqref{boclap} is given by 
\begin{equation}
\label{negimv}
\{s\in C^{\infty}(\reg(V),E)\cap L^{2}(\reg(V),E),\  \nabla s\in L^2(\reg(V),T^*\reg(V)\otimes E),\ \nabla^t(\nabla s)\in L^2(\reg(V),E)\}.
\end{equation}
The last statement is equivalent to say that  $\nabla^t \circ \nabla$, with domain given by \eqref{negimv}, is essentially self-adjoint.
\end{cor}

\begin{proof}
The assertion \eqref{muio}--\eqref{waserz} follow by Theorem \ref{skernel} and Corollary \ref{quovadis}. For \eqref{negimv} we have $(\nabla^t\circ \nabla)^{\mathcal{F}}=\nabla^t_{\max}\circ \nabla_{\min}$  by Prop. \ref{fall}. Moreover, by Prop. \ref{partiro},  we know that $C^{\infty}(\reg(V),E)\cap \mathcal{D}(\nabla^t_{\max}\circ \nabla_{\min})$ is dense in $\mathcal{D}(\nabla^t_{\max}\circ \nabla_{\min})$ with respect to its graph norm and by Theorem \ref{pianopiano} we know that $\nabla_{\max}=\nabla_{\min}$ and therefore  $\nabla_{\max}^t=\nabla_{\min}^t$. All together these propositions imply immediately \eqref{negimv}. Finally we point out that the discreteness of \eqref{boclap} follows already by Theorem \ref{pianopiano} when $v>1$. In fact we know that the inclusion $\mathcal{D}(\nabla_{\max})\hookrightarrow L^2(\reg(V),E)$ is compact where $\mathcal{D}(\nabla_{\max})$ is endowed with its graph norm. Moreover, as showed in the proof of Prop. \ref{partiro}, the inclusion $\mathcal{D}((\nabla^t\circ \nabla)^{\mathcal{F}})\hookrightarrow \mathcal{D}(\nabla_{\max})$ is continuous where again each domain is endowed with its graph norm. Therefore the inclusion $\mathcal{D}((\nabla^t\circ \nabla)^{\mathcal{F}})\hookrightarrow L^2(\reg(V),E)$ is a compact operator and this is well known to be equivalent to the discreteness of $(\nabla^t\circ \nabla)^{\mathcal{F}}:L^2(\reg(V),E)\rightarrow L^2(\reg(V),E)$.
\end{proof}

\begin{prop}
\label{pietrob}
Under the assumptions of Theorem \ref{skernel}. Assume that the complex dimension of $V$ satisfies $v>1$. Let $k_{P}(t,x,y)$ and $\|k_{P}(t,x,y)\|_{h,\op}$ be as in the proof of Theorem \ref{skernelz}. Then the following inequality holds for $0<t<1$:
\begin{equation}
\label{ube}
\|k_{P}(t,x,y)\|_{h,\op}\leq Ce^{-tc}t^{-v}.
\end{equation}
This implies that
\begin{enumerate}
\item $e^{-tP^{\mathcal{F}}}$ is a ultracontractive operator for each $0<t<1$. This means, see \cite{LSC}, that for each $0<t<1$ there exists $C_t>0$  such that   $$\|e^{-tP^{\mathcal{F}}}s\|_{L^{\infty}(\reg(V),E)}\leq C_t\|s\|_{L^1(\reg(V),E)}$$ for each $s\in L^{1}(\reg(V),E)$. In particular, for each $0<t<1$, $e^{-tP^{\mathcal{F}}}:L^{1}(\reg(V),E)\rightarrow L^{\infty}(\reg(V),E)$ is continuous.
\item If $s$ is an eigensection of $P^\mathcal{F}:L^2(\reg(V),E)\rightarrow L^2(\reg(V),E)$ then $s\in L^{\infty}(\reg(V),E)$.
\end{enumerate}
\end{prop}

\begin{proof}
This follows by Prop. \ref{pietrobz} and by Theorem \ref{skernel}.
\end{proof}

With the next result we extend Cor. 5.5 of \cite{LT} to our setting. We refer  to \cite{GHa} for the notion of degree of a complex projective variety.

\begin{teo}
\label{lower}
There exists a positive constant $\gamma=\gamma(d,n,m)$, that is $\gamma$ depends only on the dimension of the ambient space $\mathbb{C}\mathbb{P}^n$, on the degree $d$ and  on the rank $m$, such that for every irreducible complex projective variety $V\subset \mathbb{C}\mathbb{P}^n$ of degree $d$, for every vector bundle $E$ on $\reg(V)$ of rank $m$ endowed with an arbitrary metric $h$ and for every Schr\"odinger type operator $P:C^{\infty}_c(\reg(V),E)\rightarrow C^{\infty}_c(\reg(V),E)$ as in Theorem \ref{skernel} with  $L\geq 0$, the $(md)$-th eingenvalue of $P^{\mathcal{F}}$, that is $\lambda_{md}$, satisfies the following inequality:
\begin{equation}
\label{vicvic}
0<\gamma \leq \lambda_{md}. 
\end{equation}
\end{teo}

\begin{proof}
As first step we show that $\lambda_{md}\neq 0$. By \eqref{voxz} we have a continuous inclusion $\mathcal{D}(P^{\mathcal{F}})\hookrightarrow \mathcal{D}(\nabla_{\min})$ where each domain is endowed with the corresponding graph norm. Consider now $s\in \ker(P^{\mathcal{F}})$. Then, by Prop. \ref{pietrob} and by the fact that $P$ is elliptic, we have $s\in L^{\infty}(\reg(V),E)\cap C^{\infty}(\reg(V),E)$. We want to show that  $s\in \ker(\nabla_{\min})$.  Let $\{\phi_k\}_{k\in \mathbb{N}}$ be a sequence as in Prop. \ref{taglia}. Then we have: 
\begin{align}
\label{toporibattino}
& 0=P^{\mathcal{F}}s=\langle P^{\mathcal{F}}s ,s\rangle_{L^2(\reg(V),E)}=\lim_{k\rightarrow \infty} \langle P^{\mathcal{F}}s ,\phi_k^2 s\rangle_{L^2(\reg(V),E)}=\lim_{k\rightarrow \infty} \langle \nabla^t(\nabla s)+Ls ,\phi_k^2 s\rangle_{L^2(\reg(V),E)}\\ \nonumber &=\lim_{k\rightarrow \infty} \langle \nabla s ,\nabla(\phi_k^2 s)\rangle_{L^2(\reg(V),E)}+\lim_{k\rightarrow \infty} \langle L(\phi_ks) ,\phi_k s\rangle_{L^2(\reg(V),E)}\\ \nonumber &\geq \lim_{k\rightarrow \infty} \langle \nabla s ,\nabla(\phi_k^2 s)\rangle_{L^2(\reg(V),E)}=\lim_{k\rightarrow \infty} \langle \nabla s ,\phi_k^2\nabla s\rangle_{L^2(\reg(V),E)}+\lim_{k\rightarrow \infty} \langle \nabla s ,2\phi_k(d_{0,\min}\phi_k)\otimes s\rangle_{L^2(\reg(V),E)}\\ \nonumber &=\|\nabla s\|^2_{L^2(\reg(V),T^*\reg(V)\otimes E)}.
\end{align}
In the computations of the limits above we used the Dominate convergence Theorem to deduce that $$\lim_{k\rightarrow \infty} \langle \nabla s,\phi_k^2\nabla s\rangle_{L^2(\reg(V),T^*\reg(V)\otimes E)}=\|\nabla s\|^2_{L^2(\reg(V),T^*\reg(V)\otimes E)}$$ and  the inequality $$\langle \nabla s, \phi_k(d_{0,\min}\phi_k)\otimes s) \rangle_{L^2(\reg(V),T^*\reg(V)\otimes E)}\leq \|\nabla s\|_{L^2(\reg(V),T^*\reg(V)\otimes E)}\|s\|_{L^{\infty}(\reg(V),E)}\|d_{0,\min}\phi_k\|_{L^2\Omega^1(\reg(V),g)}$$ to deduce that $$\lim_{k\rightarrow \infty} \langle \nabla s ,2\phi_k(d_{0,\min}\phi_k)\otimes s\rangle_{L^2(\reg(V),E)}=0.$$ By \eqref{toporibattino}  we can thus conclude that $\nabla s=0$. This tells us that $\ker(P^{\mathcal{F}})$ is made of parallel sections and this in turn implies that $\dim(\ker(P^{\mathcal{F}}))\leq m$. Hence for the eigenvalues of $P^{\mathcal{F}}$ we have $$0\leq\lambda_0\leq \lambda_1\leq...\leq \lambda_{m-2}\leq \lambda_{m-1}<\lambda_m\leq...$$ and so we proved that $\lambda_{md}\neq 0$.\\ 
Now we proceed  adapting to our context the proof given in  \cite{LT}, Cor. 5.5.   Let $g_{FS}$ be the Fubini-Study metric on $\mathbb{C}\mathbb{P}^n$. Let us label by $\tilde{k}(t,x,y)$ the heat kernel of the unique self-adjoint extension of the  Laplacian $\Delta_0$ acting on $C^{\infty}(\mathbb{C}\mathbb{P}^n)$ that, with a little abuse of notation, we still label by $\Delta_0:L^2(\mathbb{C}\mathbb{P}^n,g_{FS})\rightarrow L^2(\mathbb{C}\mathbb{P}^n,g_{FS})$.  If we denote by $r:\mathbb{C}\mathbb{P}^n\times \mathbb{C}\mathbb{P}^n\rightarrow \mathbb{R}$ the distance function induced by the Fubini-Study metric then we know that $\tilde{k}(t,x,y)=\tilde{k}(t,r(x,y))$, see \cite{LT}.  Therefore, by Theorem \ref{skernel} and  Theorem 2.1 in \cite{LT}, we get  that $\tr(e^{-tP^{\mathcal{F}}})\leq me^{-tc}\tilde{k}(t,0)$. Moreover,  for $e^{-t\Delta_0}:L^2(\mathbb{C}\mathbb{P}^n,g_{FS})\rightarrow L^2(\mathbb{C}\mathbb{P}^n,g_{FS})$, we have 
\begin{equation}
\label{cepro}
\vol_{FS}(\mathbb{C}\mathbb{P}^n)\tilde{k}(t,0)=\int_{\mathbb{C}\mathbb{P}^n}\tilde{k}(t,0)\dvol_{g_{FS}}=\Tr(e^{-t\Delta_0})=\sum_{j=0}^{\infty}e^{-t\mu_j}\leq (1+b(n)t^{-n})
\end{equation}
 where $b(n)$ is a positive constant depending only on $n$, $\mu_j$ is the $(j+1)$-th eigenvalue of $\Delta_0:L^2(\mathbb{C}\mathbb{P}^n,g_{FS})\rightarrow L^2(\mathbb{C}\mathbb{P}^n,g_{FS})$ and $\vol_{FS}(\mathbb{C}\mathbb{P}^n)$ is the volume of $\mathbb{C}\mathbb{P}^n$ with respect to the Fubini-Study metric. In particular the last inequality in \eqref{cepro} comes from the asymptotic expansion of  $\tilde{k}(t,0)$ and holds for $t\in (0,1)$.
Hence, integrating $\tr(e^{-tP^{\mathcal{F}}})$ and $\tilde{k}(t,0)$ over $\reg(V)$,  we have:  $$\Tr(e^{-tP^{\mathcal{F}}})\leq me^{-tc}\vol_g(\reg(V))\tilde{k}(t,0)=me^{-tc}\frac{\vol_g(\reg(V))}{\vol_{FS}(\mathbb{C}\mathbb{P}^n)}(1+\sum_{j=1}^{\infty}e^{-t\mu_j})\leq me^{-tc} \frac{\vol_g(\reg(V))}{\vol_{FS}(\mathbb{C}\mathbb{P}^n)}(1+b(n)t^{-n})$$  and finally, using the fact that $\vol_g(\reg(V))=d\vol_{FS}(\mathbb{C}\mathbb{P}^n)$ \footnote{see for instance \cite{LT} pag. 876 or \cite{KIY} pag. 97}, we get 
\begin{equation}
\label{bbio}
q+\sum_{i=q}^{\infty}e^{-t\lambda_j}\leq me^{-tc} d(1+b(n)t^{-n})
\end{equation}
where $\lambda_j$ is the $(j+1)$-th eigenvalue of $P^{\mathcal{F}}$ and $q:=\dim(\ker(P^{\mathcal{F}}))$. Clearly we have $\lambda_{md}\geq \lambda_j$ for ${md}\geq j$. This in turn implies that $e^{-t\lambda_{md}}\leq e^{-t\lambda_j}$ for every $j\leq md$ and therefore $(md-q+1)e^{-t\lambda_{md}}\leq \sum_{j=q}^{md}e^{-t\lambda_j}$. 
In this way  for every  $0<\alpha<\lambda_{md}$, performing the  substitution $t=\frac{\alpha}{\lambda_{md}}$, we get from \eqref{bbio} the following inequality
\begin{equation}
\label{felicia}
q+(md-q+1)e^{-\alpha}\leq me^{-tc} d\left(1+b(n)\left(\frac{\alpha}{\lambda_{md}}\right)^{-n}\right)\leq m d\left(1+b(n)\left(\frac{\alpha}{\lambda_{md}}\right)^{-n}\right)
\end{equation}
and therefore 
$$\alpha\left(\frac{q-md+(md-q+1)e^{-\alpha}}{mdb(n)}\right)^{\frac{1}{n}}\leq \lambda_{md}$$
which in turn implies 
\begin{equation}
\label{pietrobb}
\alpha\left(\frac{-md+(md+1)e^{-\alpha}}{mdb(n)}\right)^{\frac{1}{n}}\leq\lambda_{md}.
\end{equation}
Clearly, choosing $\alpha$ sufficiently small, we have $-md+(md+1)e^{-\alpha}>0$ and hence the thesis follows defining $$\gamma:=\alpha\left(\frac{-md+(md+1)e^{-\alpha}}{mdb(n)}\right)^{\frac{1}{n}}.$$
\end{proof}

As an immediate consequence of Theorem  \ref{lower} we have the following corollary. In the case $N=\mathbb{C}$ we get Cor. 5.5 in \cite{LT}.

\begin{cor}
\label{attenti}
There exists a positive constant $\gamma'=\gamma'(d,n)$, that is $\gamma'$ depends only on the dimension of the ambient space $\mathbb{C}\mathbb{P}^n$ and on the degree $d$, such that for every irreducible complex projective variety $V\subset \mathbb{C}\mathbb{P}^n$ of degree $d$, for every Hermitian line bundle $(N,h)$ on $\reg(V)$ and for every Schr\"odinger type operator $P:C^{\infty}_c(\reg(V),N)\rightarrow C^{\infty}_c(\reg(V),N)$ as in Theorem \ref{skernel}  with $L\geq 0$ we have the following uniform lower bound for $\lambda_d$, the $d$-th eigenvalue of $P^{\mathcal{F}}$:
\begin{equation}
\label{beabei}
0<\gamma' \leq \lambda_d.
\end{equation}
\end{cor}

Finally we conclude the section with the following applications. For the definition and the general properties of Dirac operators we refer to the the books \cite{BGV}, \cite{BBW} \cite{LaM} and \cite{Jroe}.

\begin{cor}
\label{diracy}
Let $V$, $E$, $g$ and $h$ be as in Theorem \ref{pianopiano}. Assume that $E$ is a Clifford module. Let $D=\tilde{c}\circ \nabla$
\begin{equation}
\label{brose}
D:C^{\infty}_c(\reg(V),E)\rightarrow C^{\infty}_c(\reg(V),E)
\end{equation}
be a Dirac operator where $\nabla:C^{\infty}(\reg(V),E)\rightarrow C^{\infty}(\reg(V),T^*\reg(V)\otimes E)$ is a metric connection and $\tilde{c}\in \Hom(T^*\reg(V)\otimes E,E)$ is the bundle homomorphism induced by the Clifford multiplication. Let  $D^2$ be the Dirac Laplacian and let $L$ be the endomorphism of $E$ arising in the Weitzenb\"ock decomposition formula, see\cite{Jroe} pag. 43--44, 
\begin{equation}
\label{vaso}
D^2=\nabla^t\circ \nabla +L.
\end{equation}
 Assume that there is a  constant $c\in \mathbb{R}$ such that  $h( L\phi,\phi)\geq ch( \phi,\phi)$ for each $\phi\in C^{\infty}_c(\reg(V),E)$. Then  Theorem \ref{skernel}, Corollary \ref{quovadis} and Prop. \ref{pietrob} hold for $e^{-tD^{2,\mathcal{F}}}$ where $D^{2,\mathcal{F}}$ is the Friedrich extension of $D^2$. In particular $$e^{-tD^{2,\mathcal{F}}}:L^2(\reg(V),E)\rightarrow L^2(\reg(V),E)$$ is a trace class operator and $D^{2,\mathcal{F}}:L^2(\reg(V),E)\rightarrow L^2(\reg(V),E)$ is a discrete operator.
\end{cor}

As application of the previous theorem we have the following corollary. 
\begin{cor}
\label{spino}
Let $V$ and $g$ be as in  Theorem \ref{pianopiano}. Assume that $\reg(V)$ is a spin manifold and assume moreover that $s_g$, the scalar curvature of $g$, satisfies $s_g\geq c$ for some $c\in \mathbb{R}$. Let $\Sigma$ be the spinor bundle on $\reg(V)$ and let 
\begin{equation}
\label{spinn}
\eth:C^{\infty}_c(\reg(V),\Sigma)\rightarrow C^{\infty}_c(\reg(V),\Sigma)
\end{equation}
be  the associated spin Dirac operator. Then Theorem \ref{skernel}, Corollary \ref{quovadis} and Prop. \ref{pietrob} hold for \begin{equation}
\label{azzs}
e^{-t\eth^{2,\mathcal{F}}}:L^2(\reg(V),\Sigma)\rightarrow L^2(\reg(V),\Sigma).
\end{equation}
In particular \eqref{azzs} is a trace class operator and $\eth:L^2(\reg(V),\Sigma)\rightarrow L^2(\reg(V),\Sigma)$ is a discrete operator.
\end{cor}

\begin{proof}
It is a consequence of the Lichnerowicz formula, $\eth^2=\nabla^t\circ \nabla+\frac{1}{4}s_g$, see for instance \cite{LaM} pag. 160.
\end{proof}

\begin{cor}
\label{sitornaacasa}
Let $V\subset \mathbb{C}\mathbb{P}^n$ be an irreducible complex projective variety of complex dimension $v$. Let $g$ be the K\"ahler metric on $\reg(V)$ induced by the Fubini Study metric of $\mathbb{C}\mathbb{P}^n$. Let  $k\in \{0,...,2v\}$ and consider the Bochner-Weitzenb\"ock identity for the Laplacian $\Delta_k:\Omega_c^k(\reg(V))\rightarrow \Omega_c^k(\reg(V))$, see \cite{LaM} pag. 155 or \cite{Jroe} pag. 43--44,
\begin{equation}
\label{holynight}
\Delta_k=\nabla^t_k\circ \nabla_k+L_k,
\end{equation}
where $\nabla_k:\Omega^k(\reg(V))\rightarrow C^{\infty}(\reg(V),T^*\reg(V)\otimes \Lambda^kT^*\reg(V))$ is the metric connection induced by the Levi Civita conneection. Assume that there is a constant $c$ such that $\langle L_k\eta,\eta\rangle_{g_k}\geq c\langle\eta,\eta\rangle_{g_k}$ for each $\eta\in \Omega^k_c(\reg(V))$. Let $\Delta_k^{\mathcal{F}}$ be the Friedrich extension of $\Delta_k$ and let 
\begin{equation}
\label{nununu}
e^{-t\Delta_k^{\mathcal{F}}}:L^2\Omega^k(\reg(V),g)\rightarrow L^2\Omega^k(\reg(V),g)
\end{equation}
be the heat operator associated to $\Delta_k^{\mathcal{F}}$. Then \eqref{nununu} is a trace class operator. In particular Theorem \ref{skernel}, Cor. \ref{quovadis} and Prop. \ref{pietrob} hold for \eqref{nununu}.\\ Consider now the Hodge-Kodaira Laplacian $\Delta_{p,q,\overline{\partial}}:\Omega_c^{p,q}(\reg(V))\rightarrow \Omega_c^{p,q}(\reg(V))$ such that $p+q=k$. Let  $\Delta_{p,q,\overline{\partial}}^{\mathcal{F}}$ be the Friedrich extension of $\Delta_{p,q,\overline{\partial}}$ and let 
\begin{equation}
\label{nununuza}
e^{-t\Delta_{p,q,\overline{\partial}}^{\mathcal{F}}}:L^2\Omega^{p,q}(\reg(V),g)\rightarrow L^2\Omega^{p,q}(\reg(V),g)
\end{equation}
be the heat operator associated to $\Delta_{p,q,\overline{\partial}}^{\mathcal{F}}$. Then \eqref{nununuza} is a trace class operator. As in the previous case Theorem \ref{skernel}, Cor. \ref{quovadis} and Prop. \ref{pietrob} hold for \eqref{nununuza}.
\end{cor}

\begin{proof}
The first part of the theorem, that is the one concerning with $\Delta_{k}^{\mathcal{F}}$, is an immediate application of Theorem \ref{skernel}, Cor. \ref{quovadis} and Prop. \ref{pietrob}. The second part follows by the fact that $\Delta_k=2\bigoplus_{p+q=k}\Delta_{p,q,\overline{\partial}}$ and that, see for instance \cite{FB} pag. 169,  $\Delta_k^{\mathcal{F}}=2\bigoplus_{p+q=k}\Delta_{p,q,\overline{\partial}}^{\mathcal{F}}$. 
\end{proof}

\section{Applications to stratified pseudomanifolds}

This last section contains applications concerning Thom-Mather stratified pseudomanifolds. We start  recalling briefly the basic definitions and properties.  We first recall that, given a topological space $Z$, $C(Z)$ stands for the cone over $Z$ that is $Z\times [0,2)/\sim$ where $(p,t)\sim (q,r)$ if and only if $r=t=0$.

\begin{defi}   
\label{thom}
A smoothly Thom-Mather stratified pseudomanifold $X$  of dimension $m$  is a metrizable, locally compact, second countable space which admits a locally finite decomposition into a union of locally closed strata $\mathfrak{G}=\{Y_{\alpha}\}$, where each $Y_{\alpha}$ is a smooth, open and connected manifold, with dimension depending on the index $\alpha$. We assume the following:
\begin{enumerate}
\item[(i)] If $Y_{\alpha}$, $Y_{\beta} \in \mathfrak{G}$ and $Y_{\alpha} \cap \overline{Y}_{\beta} \neq \emptyset$ then $Y_{\alpha} \subset \overline{Y}_{\beta}$
\item[(ii)]  Each stratum $Y$ is endowed with a set of control data $T_{Y} , \pi_{Y}$ and $\rho_{Y}$ ; here $T_{Y}$ is a neighborhood of $Y$ in $X$ which retracts onto $Y$, $\pi_{Y} : T_{Y} \rightarrow Y$
is a fixed continuous retraction and $\rho_{Y}: T_{Y}\rightarrow [0, 2)$ is a continuous function in this tubular neighborhood such that $\rho_{Y}^{-1}(0) = Y$ . Furthermore,
we require that if $Z \in \mathfrak{G}$ and $Z \cap T_{Y}\neq \emptyset$  then
$(\pi_{Y} , \rho_{Y} ) : T_{Y} \cap Z \rightarrow Y\times [0,2) $
is a proper smooth submersion.
\item[(iii)] If $W, Y,Z \in \mathfrak{G}$, and if $p \in T_{Y} \cap T_{Z} \cap W$ and $\pi_{Z}(p) \in T_{Y} \cap Z$ then
$\pi_{Y} (\pi_{Z}(p)) = \pi_{Y} (p)$ and $\rho_{Y} (\pi_{Z}(p)) = \rho_{Y} (p)$.
\item[(iv)] If $Y,Z \in \mathfrak{G}$, then
$Y \cap \overline{Z} \neq \emptyset \Leftrightarrow T_{Y} \cap Z \neq \emptyset$ ,
$T_{Y} \cap T_{Z} \neq \emptyset \Leftrightarrow Y\subset \overline{Z}, Y = Z\ or\ Z\subset \overline{Y} .$
\item[(v)]  For each $Y \in \mathfrak {G}$, the restriction $\pi_{Y} : T_{Y}\rightarrow Y$ is a locally trivial fibration with fibre the cone $C(L_{Y})$ over some other stratified space $L_{Y}$ (called the link over $Y$ ), with atlas $\mathcal{U}_{Y} = \{(\phi,\mathcal{U})\}$ where each $\phi$ is a trivialization
$\pi^{-1}_{Y} (U) \rightarrow U \times C(L_{Y} )$, and the transition functions are stratified isomorphisms  which preserve the rays of each conic
fibre as well as the radial variable $\rho_{Y}$ itself, hence are suspensions of isomorphisms of each link $L_{Y}$ which vary smoothly with the variable $y\in U$.
\item[(vi)] For each $j$ let $X_{j}$ be the union of all strata of dimension less or equal than $j$, then 
$$
X_{m-1}=X_{m-2}\ \text{and  $X\setminus X_{m-2}$ dense\ in $X$}
$$
\end{enumerate}
\end{defi}

The \emph{depth} of a stratum $Y$ is largest integer $k$ such that there is a chain of strata $Y=Y_{k},...,Y_{0}$ such that $Y_{j}\subset \overline{Y_{j-1}}$ for $1\leq j\leq k.$ A stratum of maximal depth is always a closed subset of $X$.  The  maximal depth of any stratum in $X$ is called the \emph{depth of $X$} as stratified spaces.
 Consider the filtration
\begin{equation}
X = X_{m}\supset X_{m-1}= X_{m-2}\supset X_{m-3}\supset...\supset X_{0}.
\label{pippo}
\end{equation}
 We refer to the open subset $X\setminus X_{m-2}$ of a smoothly Thom-Mather-stratified pseudomanifold $X$ as its regular set, and the union of all other strata as the singular set,
$$\reg(X):=X\setminus \sing(X)\ \text{where}\ \sing(X):=\bigcup_{Y\in \mathfrak{G}, \depth(Y)>0 }Y. $$
Given two Thom-Mather smoothly stratified pseudomanifolds  $X$ and $X'$,  a stratified isomorphism between them is a
homeomorphism $F: X\rightarrow X'$ which carries the open strata of $X$ to the open strata of $X'$
diffeomorphically, and such that $\pi'_{F(Y )}\circ  F = F \circ \pi_Y$ , $\rho'_{F(Y)}\circ F=\rho_Y$  for all $Y\in \mathfrak{G}(X)$.
For more details, properties and comments we refer to \cite{ALMP}, \cite{BHS}, \cite{RrHS} and the bibliography cited there. Here we point out that a large class of topological space such as irreducible complex analytic spaces or  quotient of manifolds through a proper Lie group action belong to this class of spaces.\\ 
As a next step we introduce the class of smooth Riemmanian metrics on $\reg(X)$ which we are interested in. The definition is given by induction on the depth of $X$.  We label by $\hat{c}:=(c_2,...,c_m)$ a $(m-1)$-tuple of non negative real numbers.

\begin{defi}
\label{iter}
Let $X$ be a smoothly Thom-Mather-stratified pseudomanifold and let $g$ be a Riemannian metric on $\reg(X)$. If $\depth(X)=0$, that is $X$ is a smooth manifold, a  $\hat{c}$-iterated  edge metric is understood to be any smooth Riemannian metric on $X$. Suppose now that $\depth(X)=k$ and that the definition of $\hat{c}$-iterated  edge metric is given in the case $\depth(X)\leq k-1$; then we call a smooth Riemannian metric $g$ on $\reg(X)$ a  $\hat{c}$-iterated  edge metric  if it satisfies the following properties:
\begin{itemize}
\item Let $Y$ be a stratum of $X$ such that $Y\subset X_{i}\setminus  X_{i-1}$; by  Def. \ref{thom} for each $q\in Y$ there exist an open neighbourhood $U$ of $q$ in $Y$ such that 
$$
\phi:\pi_{Y}^{-1}(U)\longrightarrow U\times C(L_{Y})
$$
 is a stratified isomorphism; in particular, 
$$
\phi:\pi_{Y}^{-1}(U)\cap \reg(X)\longrightarrow U\times \reg(C(L_{Y}))
$$
is a smooth diffeomorphism. Then, for each $q\in Y$, there exists one of these trivializations $(\phi,U)$ such that $g$ restricted on $\pi_{Y}^{-1}(U)\cap \reg(X)$ satisfies the following properties:
\begin{equation} 
(\phi^{-1})^{*}(g|_{\pi_{Y}^{-1}(U)\cap \reg(X)})\sim  dr^2+h_{U}+r^{2c_{m-i}}g_{L_{Y}}
\label{yhn}
\end{equation}
 where $m$ is the dimension of $X$, $h_{U}$ is a Riemannian metric defined over $U$ and  $g_{L_{Y}}$ is a $(c_2,...,c_{m-i-1})$-iterated edge metric  on $\reg(L_{Y})$, $d r^2+h_{U}+r^{2c_{m-i}}g_{L_{Y}}$ is a Riemannian metric of product type on $U\times \reg(C(L_{Y}))$ and with $\sim$ we mean quasi-isometric. 
\end{itemize}
\end{defi} 
We remark that in \eqref{yhn} the neighborhood  $U$ can be chosen sufficiently small so that it is diffeomorphic to $(0,1)^i$ and $h_U$ it is quasi-isometric to the Euclidean metric restricted on $(0,1)^i$. Moreover we point out that with this kind of Riemannian metrics we have $\vol_g(\reg(X))<\infty$ in case $X$ is compact. There is the following nontrivial existence result:

\begin{prop}
Let $X$ be a smoothly Thom-Mather stratified pseudomanifold of dimension $m$. For any $(m-1)$-tuple  of positive numbers $\hat{c}=(c_2,...,c_m)$, there exists  a smooth Riemannian metric on $\reg(X)$ which is a $\hat{c}$-iterated edge metric.
\end{prop}
\begin{proof}
See \cite{BHS} or \cite{ALMP} in the case $\hat{c}=(1,...,1,...,1)$.
\end{proof}
\noindent When $\hat{c}=(1,...,1)$ we will call this kind of  metrics simply  {\em iterated edge metric}.\\
The importance of this class of metrics lies on its deep connection with the topology of $X$. In fact, as pointed out by Cheeger in his seminal paper \cite{JC} (see also  \cite{FBe} and the bibliography cited there  for further developments ) the $L^2$-cohomology of $\reg(X)$ associated to an iterated edge metric is isomorphic to the intersection cohomology  of $X$ associated with a perversity depending only on $\hat{c}$. In other words the $L^2$-cohomology of these kind of metrics (which a priori is an object that lives only  on $\reg(X)$) provides non trivial topological informations of the whole space $X$.

\noindent Now we have the following proposition which assure the existence of a suitable sequence of cut-off functions. 

\begin{prop}
\label{tagliaz}
Let  $X$ be a compact, smoothly Thom-Mather stratified pseudomanifold of dimension $m$. Consider on $\reg(X)$ an iterated edge metric $g$. Then there exists a sequence of Lipschitz  functions with compact support contained in $\reg(X)$, $\{\phi_j\}_{j\in \mathbb{N}}$, such that 
\begin{itemize}
\item  $0\leq \phi_j\leq 1$ for each $j$.
\item $\phi_j\rightarrow 1 $ pointwise.
\item $\phi_j\in \mathcal{D}(d_{0,\min})$ for each $j\in \mathbb{N}$ and  $\lim_{j\rightarrow \infty}\|d_{0,\min}\phi_j\|_{L^2\Omega^1(\reg(V),g)}=0$.  
\end{itemize}
In particular $1\in \mathcal{D}(d_{0,\min})$.
\end{prop}

\begin{proof}
See  \cite{BeGu}.
\end{proof}

\begin{teo}
\label{pianopianof}
Let $X$ be a compact, smoothly Thom-Mather stratified pseudomanifold of dimension $m$.  Consider on $\reg(X)$ an iterated edge metric $g$.  Let $E$ be a vector bundle over $\reg(X)$ and let $h$ be a metric on $E$, Riemannian if $E$ is a real vector bundle, Hermitian if $E$ is a complex vector bundle.  Finally let $\nabla:C^{\infty}(\reg(X),E)\rightarrow C^{\infty}(\reg(V),T^*\reg(X)\otimes E)$ be a metric connection. We have the following properties:
\begin{itemize}
\item $W^{1,2}(\reg(X),E)=W^{1,2}_0(\reg(X),E)$.
\item Assume that $m>2$. Then there exists a continuous inclusion $W^{1,2}(\reg(X),E)\hookrightarrow L^{\frac{2m}{m-2}}(\reg(X),E)$.
\item Assume that $m>2$. Then the inclusion $W^{1,2}(\reg(X),E)\hookrightarrow L^2(\reg(X),E)$ is a compact operator.
\end{itemize}
\end{teo}

\begin{proof}
The first point follows by Prop. \ref{tagliaz} and by  Prop. \ref{sette}.  The continuous inclusion 
$W^{1,2}_0(\reg(X),g)\hookrightarrow L^{\frac{2m}{m-2}}(\reg(M),g)$ is established in \cite{ACM} Prop. 2.2. By the first point of this theorem  we know that $W^{1,2}(\reg(X),g)=W_0^{1,2}(\reg(X),g)$ and therefore we have the continuous inclusion $W^{1,2}(\reg(X),g)\hookrightarrow L^{\frac{2m}{m-2}}(\reg(X),g)$. Now, using Prop. \ref{devo}, we get the continuous inclusion $$C^{\infty}(\reg(X),E)\cap W^{1,2}(\reg(X),E)\hookrightarrow L^{\frac{2m}{m-2}}(\reg(X),E).$$ Finally, by the density of $C^{\infty}(\reg(X),E)\cap W^{1,2}(\reg(X),E)$ in $W^{1,2}(\reg(X),E)$, see Prop. \ref{carmina}, the continuous inclusion $W^{1,2}(\reg(X),E)\hookrightarrow L^{\frac{2m}{m-2}}(\reg(X),E)$ is established. Finally the third point is a consequence of the second point and  Prop. \ref{otto}.
\end{proof}

\begin{cor}
\label{delidaf}
Under the assumptions of Theorem \ref{pianopianof}. Then $\im(\nabla_{\min})=\im(\nabla_{\max})$ is a closed subspace of $L^2(\reg(X),T^*\reg(X)\otimes E)$.
\end{cor}

\begin{proof}
According to Prop. \ref{pianopianof} we know that $\nabla_{\max}=\nabla_{\min}$ and therefore $\im(\nabla_{\max})=\im(\nabla_{\min})$. Now the thesis follows by Cor. \ref{gaber}.
\end{proof}

We have also the following application.

\begin{prop}
\label{dieci}
Let $(\reg(X),g)$ be as in Theorem \ref{pianopianof}. Let  $E$ and $F$ be two vector bundles over $\reg(X)$ endowed respectively with  metrics $h$ and $\rho$, Riemannian if $E$ and $F$ are real vector bundles, Hermitian if $E$ and $F$ are complex vector bundles. Finally let  $\nabla:C^{\infty}(\reg(X),E)\rightarrow C^{\infty}(\reg(X),T^*\reg(X)\otimes E)$
be a metric connection. Consider a  first order differential operator of this type:
\begin{equation}
\label{gazzzzf}
D:=\theta_0\circ \nabla:C^{\infty}_c(\reg(X),E)\rightarrow C^{\infty}_c(\reg(X),F)
\end{equation}
where $\theta_0\in  C^{\infty}(\reg(X),\Hom(T^*\reg(X)\otimes E,F)).$ Assume that  $\theta_0$ extends as a bounded operator  $$ \theta:L^2(\reg(X), T^*\reg(X)\otimes E)\rightarrow L^2(\reg(X), F).$$  Then we have the following inclusion:
\begin{equation}
\label{cantianobellof}
\mathcal{D}(D_{\max})\cap L^{\infty}(\reg(X),E)\subset \mathcal{D}(D_{\min}).
\end{equation}
In particular \eqref{cantianobellof} holds when $D$ is the de Rham differential $d_k:\Omega_{c}^k(\reg(X))\rightarrow \Omega^{k+1}_c(\reg(X))$ or  a Dirac type  operator   $D:C^{\infty}_c(\reg(X),E)\rightarrow C^{\infty}_c(\reg(X),E)$.
\end{prop}

\begin{proof}
This follows by Theorem \ref{pianopianof} and Prop. \ref{drivevz}.
\end{proof}

\noindent Finally consider  again the setting of Theorem \ref{pianopianof}. The remaining part of this section collects applications to some  Schr\"odinger type operators
\begin{equation}
\label{polvere}
\nabla^t\circ\nabla +L 
\end{equation}
 where $\nabla: C_c^{\infty}(\reg(X),E)\rightarrow C_c^{\infty}(\reg(X),T^*M\otimes E)$ is a metric connection, $\nabla^t: C_c^{\infty}(\reg(X),T^*M\otimes E)\rightarrow C_c^{\infty}(\reg(X),E)$ is the formal adjoint of $\nabla$ and $L\in C^{\infty}(\reg(X),\End(E))$ is a bundle homomorphism.

\begin{teo}
\label{skernelf}
Let $X$, $E$, $g$, $h$, and $\nabla$ be as described above. Let $$P:=\nabla^t\circ \nabla +L,\ P:C^{\infty}_c(\reg(X),E)\rightarrow C_c^{\infty}(\reg(X),E)$$ be a Schr\"odinger type operator with $L\in C^{\infty}(\reg(X),\End(E))$. Assume that:
\begin{itemize}
\item $P$ is symmetric and  positive.
\item    There is a constant $c\in \mathbb{R}$ such that, for each $s\in C^{\infty}(\reg(X),E)$, we have  $$h(Ls,s)\geq ch(s,s).$$
\end{itemize}
Let $P^{\mathcal{F}}:L^2(\reg(X),E)\rightarrow L^2(\reg(X),E)$ be the Friedrich extension of $P$ and let  $\Delta_0^{\mathcal{F}}:L^2(\reg(X),g)\rightarrow L^2(\reg(X),g)$  be the Friedrich extension of $\Delta_0:C^{\infty}_c(\reg(X))\rightarrow C^{\infty}_c(\reg(X))$. Then the heat operator associated to $P^{\mathcal{F}}$ $$e^{-tP^{\mathcal{F}}}:L^{2}(\reg(X),E)\longrightarrow L^2(\reg(X),E)$$  
is a trace class operator and its trace satisfies the following inequality: 
\begin{equation}
\label{marzf}
\Tr(e^{-tP^{\mathcal{F}}})\leq re^{-tc}\Tr(e^{-t\Delta_0^{\mathcal{F}}})
\end{equation}
where $r$ is the rank of the vector bundle $E$.
\end{teo}

\begin{proof}
This follows by Prop. \ref{skernelz} and by Theorem \ref{pianopianof}.
\end{proof}

Analogously to the previous section we have now the following corollaries.

\begin{cor}
\label{otino}
Under the assumptions of Theorem \ref{skernelf}. For each $t\in (0,1)$ we have the following inequalities for the pointwise trace and for the heat trace of $e^{-tP^{\mathcal{F}}}$ respectively:
$$\tr(k_{P}(t,x,x))\leq rCe^{-tc}t^{\frac{-m}{2}}$$ $$\Tr(e^{-tP^{\mathcal{F}}})\leq rCe^{-tc}\vol_g(\reg(X))t^{\frac{-m}{2}}$$ where $C$ is the  positive constant arising from Prop. \ref{luglioaberlino}.
The operator $P^{\mathcal{F}}:L^2(\reg(X),E)\rightarrow L^2(\reg(X),E)$ is a discrete operator. 
If we label its eigenvalues with $$0\leq \lambda_{0}\leq \lambda_1\leq...\leq \lambda_n\leq...$$ then there exists a  positive constant $K$ such that we have the following asymptotic inequality $$\lambda_j\geq Kj^{\frac{2}{m}}+c$$  as $j\rightarrow \infty$.
\end{cor}

\begin{proof}
This follows by Cor. \ref{eigen}.
\end{proof}

\begin{cor}
\label{democ}
Let $X, E, h$ and $g$ as in the statement of  Theorem \ref{skernelf}. Consider the Bochner Laplacian $$\nabla^t\circ \nabla:C^{\infty}_c(\reg(X),E)\rightarrow C^{\infty}_c(\reg(X),E).$$ Let 
\begin{equation}
\label{boclapf}
(\nabla^t\circ \nabla)^{\mathcal{F}}:L^2(\reg(X),E)\rightarrow L^2(\reg(X),E)
\end{equation}
be its Friedrich extension. Then 
\begin{equation}
\label{muiof}
e^{-t(\nabla^t\circ \nabla)^{\mathcal{F}}}:L^{2}(\reg(X),E)\longrightarrow L^2(\reg(X),E)
\end{equation}
is a trace class operator; for every $t\in (0,1)$ its pointwise trace and its  trace satisfy the following inequalities:
\begin{equation}
\label{jfrazierf}
\tr(k_{\nabla^t\circ \nabla}(t,x,x))\leq rCt^{\frac{-m}{2}}
\end{equation} 
 \begin{equation}
\label{marzlf}
\Tr(e^{-t(\nabla^t\circ \nabla)^{\mathcal{F}}})\leq rC\vol_g(\reg(X))t^{\frac{-m}{2}}
\end{equation}
where $C$ is the positive constant arising from Prop. \ref{luglioaberlino}.
\eqref{boclapf} is a discrete operator and its sequence of eigenvalues, $\{\lambda_j\}$,  satisfies the following asymptotic inequality:
\begin{equation}
\label{waserzf}
\lambda_j\geq  Kj^{\frac{2}{m}}
\end{equation}
as $j\rightarrow \infty$ where $K$ is a positive constant.
Finally a  core domain for \eqref{boclapf} is given by 
\begin{equation}
\label{negimvf}
\{s\in C^{\infty}(\reg(X),E)\cap L^{2}(\reg(X),E):\  \nabla s\in L^2(\reg(X),T^*\reg(X)\otimes E)\ \text{and}\ \nabla^t(\nabla s)\in L^2(\reg(X),E)\}.
\end{equation}
The last statement is equivalent to say that  $\nabla^t \circ \nabla$, with domain given by \eqref{negimvf}, is essentially self-adjoint.
\end{cor}

\begin{proof}
The proof of \eqref{negimvf} is analogous to that we have provided for \eqref{negimv}. The remaining points are consequences of Cor. \ref{otino}.
\end{proof}

\begin{prop}
\label{pietrobf}
Under the assumptions of Theorem \ref{skernelf}. Assume that  $m>1$. Let $k_{P}(t,x,y)$ and $\|k_{P}(t,x,y)\|_{h,\op}$ be as in the proof of Theorem \ref{skernelz}. Then the following inequality holds for $0<t<1$:
\begin{equation}
\label{ube}
\|k_{P}(t,x,y)\|_{h,\op}\leq Ce^{-tc}t^{\frac{-m}{2}}.
\end{equation}
This implies that:
\begin{enumerate}
\item $e^{-tP^{\mathcal{F}}}$ is a ultracontractive operator for each $0<t<1$. This means, see \cite{LSC}, that for each $0<t<1$ there exists $C_t>0$  such that   $$\|e^{-tP^{\mathcal{F}}}s\|_{L^{\infty}(\reg(X),E)}\leq C_t\|s\|_{L^1(\reg(X),E)}$$ for each $s\in L^{1}(\reg(X),E)$. In particular, for each $0<t<1$, $e^{-tP^{\mathcal{F}}}:L^{1}(\reg(V),E)\rightarrow L^{\infty}(\reg(X),E)$ is continuous.
\item If $s$ is an eigensection of $P^\mathcal{F}:L^2(\reg(X),E)\rightarrow L^2(\reg(X),E)$ then $s\in L^{\infty}(\reg(X),E)$.
\end{enumerate}
\end{prop}

\begin{proof}
This follows by Prop. \ref{pietrobz} and by Theorem \ref{skernelf}.
\end{proof}

Analogously to the previous section we have the following applications to Dirac operators. 

\begin{cor}
\label{diracys}
Let $X$, $E$, $g$ and $h$ be as in Theorem \ref{pianopianof}. Assume that $E$ is a Clifford module. Let $D=\tilde{c}\circ \nabla$
\begin{equation}
\label{broses}
D:C^{\infty}_c(\reg(X),E)\rightarrow C^{\infty}_c(\reg(X),E)
\end{equation}
be a Dirac operator where $\nabla:C^{\infty}(\reg(X),E)\rightarrow C^{\infty}(\reg(X),T^*\reg(X)\otimes E)$ is a metric connection and $\tilde{c}\in \Hom(T^*\reg(X)\otimes E,E)$ is the bundle homomorphism induced by the Clifford multiplication. Let  $D^2$ be the Dirac Laplacian and let $L$ be the endomorphism of $E$ arising in the Weitzenb\"ock decomposition formula, see\cite{Jroe} pag. 43--44, 
\begin{equation}
\label{vasos}
D^2=\nabla^t\circ \nabla +L.
\end{equation}
 Assume that there is a  constant $c\in \mathbb{R}$ such that  $h( L\phi,\phi)\geq ch( \phi,\phi)$ for each $\phi\in C^{\infty}_c(\reg(V),E)$. Then  Theorem \ref{skernelf}, Corollary \ref{otino} and Prop. \ref{pietrobf} hold for $e^{-tD^{2,\mathcal{F}}}$ where $D^{2,\mathcal{F}}$ is the Friedrich extension of $D^2$. In particular $$e^{-tD^{2,\mathcal{F}}}:L^2(\reg(X),E)\rightarrow L^2(\reg(X),E)$$ is a trace class operator and $D^{2,\mathcal{F}}:L^2(\reg(X),E)\rightarrow L^2(\reg(X),E)$ is a discrete operator.
\end{cor}

\begin{cor}
\label{spinos}
Let $X$ and $g$ be as in  Theorem \ref{pianopianof}. Assume that $\reg(X)$ is a spin manifold and assume moreover that $s_g$, the scalar curvature of $g$, satisfies $s_g\geq c$ for some $c\in \mathbb{R}$. Let $\Sigma$ be the spinor bundle on $\reg(X)$ and let 
\begin{equation}
\label{spinns}
\eth:C^{\infty}_c(\reg(X),\Sigma)\rightarrow C^{\infty}_c(\reg(X),\Sigma)
\end{equation}
be  the associated spin Dirac operator. Then Theorem \ref{skernelf}, Corollary \ref{otino} and Prop. \ref{pietrobf} hold for \begin{equation}
\label{azzss}
e^{-t\eth^{2,\mathcal{F}}}:L^2(\reg(X),\Sigma)\rightarrow L^2(\reg(X),\Sigma).
\end{equation}
In particular \eqref{azzss} is a trace class operator and $\eth:L^2(\reg(X),\Sigma)\rightarrow L^2(\reg(X),\Sigma)$ is a discrete operator.
\end{cor}

\begin{cor}
\label{sitornaacasas}
Let $X$ and $g$ be as in Theorem \ref{pianopianof}. Let  $k\in \{0,...,m\}$ and consider the Bochner-Weitzenb\"ock identity for the Laplacian $\Delta_k:\Omega_c^k(\reg(X))\rightarrow \Omega_c^k(\reg(X))$, see \cite{LaM} pag. 155 or \cite{Jroe} pag. 43--44,
\begin{equation}
\label{holynights}
\Delta_k=\nabla^t_k\circ \nabla_k+L_k,
\end{equation}
where $\nabla_k:\Omega^k(\reg(X))\rightarrow C^{\infty}(\reg(X),T^*\reg(X)\otimes \Lambda^kT^*\reg(X))$ is the metric connection induced by the Levi Civita connection. Assume that there is a constant $c$ such that $\langle L_k\eta,\eta\rangle_{g_k}\geq c\langle\eta,\eta\rangle_{g_k}$ for each $\eta\in \Omega^k_c(\reg(X))$. Let $\Delta_k^{\mathcal{F}}$ be the Friedrich extension of $\Delta_k$ and let 
\begin{equation}
\label{nununus}
e^{-t\Delta_k^{\mathcal{F}}}:L^2\Omega^k(\reg(X),g)\rightarrow L^2\Omega^k(\reg(X),g)
\end{equation}
be the heat operator associated to $\Delta_k^{\mathcal{F}}$. Then \eqref{nununus} is a trace class operator. In particular Theorem \ref{skernelf}, Cor. \ref{otino} and Prop. \ref{pietrobf} hold for \eqref{nununu}.
\end{cor}

\end{document}